\documentclass[11pt]{article}

\input{macrogilles_main.tex}
\usepackage{amsmath,amsfonts,amsthm}
\usepackage[isolatin]{inputenc}
\usepackage[round]{natbib}
\usepackage[breaklinks]{hyperref}
\usepackage{eucal}

\usepackage[ruled,algosection]{algorithm2e}

\usepackage{geometry}
\newcommand{\fst}{f^*}
\newcommand{\Y}{\Upsilon}
\newcommand{\fm}{f_m}
\newcommand{\alm}{\wh{\alpha}_m}
\newcommand{\xini}{\xi_{n,i}}
\newcommand{\eni}{e_{n,i}}
\newcommand{\blambda}{\wt{\lambda}}
\newcommand{\nux}{\nu}

\newcommand{\wY}{\ol{\Y}}

\begin{document}

\title{Convergence rates of Kernel Conjugate Gradient for random design regression}

\author{Gilles Blanchard\thanks{This research was partly supported by the DFG via Research Unit 1735 {\em Structural Inference in Statistics}.}
\\
Mathematics Institute, University of Potsdam\\
Karl-Liebknecht-Stra{\ss}e 24-25\\ 14476 Potsdam, Germany\\
\texttt{blanchard@math.uni-potsdam.de}\\
\and
Nicole Kr\"amer \\
Staburo GmbH\\
Aschauer Str. 30a,
81549 M\"unchen\\
\texttt{kraemer@staburo.de}
}
\date{}

\maketitle

\begin{abstract}  We prove statistical rates of convergence for kernel-based
least squares regression from i.i.d. data using a conjugate gradient algorithm,
where regularization against overfitting is obtained by early stopping.
This method is related to Kernel Partial Least Squares, a
 regression method that combines supervised dimensionality reduction with least squares projection.
 Following the setting introduced in earlier related literature,
 we study so-called ``fast convergence rates'' depending
 on the regularity of the target regression function (measured by a source condition
 in terms of the kernel integral operator) and
 on the effective dimensionality of the data
 mapped into the kernel space. We obtain upper bounds,
 essentially matching known minimax lower bounds, for the $\cL^2$ (prediction) norm
 as well as for the stronger Hilbert norm, if  the
 true regression function belongs to the reproducing kernel Hilbert space.
 If the latter assumption is not fulfilled, we obtain similar convergence rates
 for appropriate norms, provided additional unlabeled data are available.
\end{abstract}

\section{Introduction}

\subsection{Setting}

Consider the nonparametric random design regression (also
called ``statistical learning'') problem, where an
$n$-sample of observations $(X_i,Y_i)\in \cX \times \mbr, 1 \leq i
\leq n$ is assumed to be drawn i.i.d. from an unknown distribution
$P$\,.  Here and in the rest of this work, $\cX$ is assumed to be a
Radon space, for instance an open subset of $\mbr^d$\,.  The goal is the
estimation of the regression function $\fst(x) := \ee{(X,Y)\sim
  P}{Y|X=x}$\,; it is assumed that the true regression function $\fst$
belongs to the space $\mathcal{L}^2(\nux)$ of square-integrable
functions ($\nux$ denotes the $X$-marginal of $P$ on the space $\cX$).

If $\wh{f}$ is an estimator of $\fst$\,, its quality is measured
via the $\mathcal{L}^2(\nux)$ distance,
\begin{equation}
\label{eq:l2dist}
\norm{\wh{f} - \fst}^2_{2,\nux} = \ee{X\sim \nux}{(\wh{f}(X)-\fst(X))^2}.
\end{equation}
This distance is natural for random design regression, since, if we
interpret this setting as a prediction problem for a new independent example
$(X,Y)\sim P$ where the quality of prediction is measured by the squared
error loss $\ell(f,x,y)=(f(x)-y)^2$\,, then it is well-known that $\fst$
is the minimizer of the average prediction (or generalization) error
$\cE(f,P)=\ee{(X,Y)\sim P}{(f(X)-Y)^2}$ over all squared integrable
functions, and that the above distance coincides with the excess
prediction error:
\[
\norm{\wh{f} - \fst}^2_{2,\nux} = \cE(\wh{f}) - \cE(\fst)\,. 
\]

Assume that $k: \cX^2 \rightarrow \mbr$ is a real-valued {\em reproducing kernel} on the space $\cX$\,,
with associated reproducing kernel Hilbert space $\cH$\,.
The well-established principle of non-parametric estimation by reproducing kernel methods consists in
considering estimators admitting a {\em kernel expansion} of the form
\begin{equation}
\label{eq:kernexp}
\wh{f} = f_\wh{\alpha}(x) := \frac{1}{n} \sum_{i=1}^n \wh{\alpha}^{(i)} k(X_i,x)\,,
\end{equation}
where the real coefficients $\wh{\alpha}^{(i)}, 1\leq i \leq n$ are determined from the data
(see for example \citealp{CriSha04} and \citealp{SteChr08} for comprehensive references on the topic.)
To avoid some confusion, we point out that the normalization by $n^{-1}$ that we use here 
in the kernel expansion is not present in most references on the subject, but we find it 
technically convenient.

We denote by $K_n=\frac{1}{n}(k(X_i,X_j))_{i,j}\in \mathbb{R}^{n\times n}$
the normalized kernel matrix and by $\Y=(Y_1,\ldots,Y_n)^\top \in
\mathbb{R}^n$ the $n$-vector of response observations. A naive approach to determining the
vector of kernel expansion coefficients $\wh{\alpha}$ is to choose those in order that $f_{\wh{\alpha}}(X_i)=Y_i$
holds for all $i=1,\ldots,n$\,, that is, solving the linear equation
\begin{equation}
\label{eq:kernel_ls}
K_n \alpha =\Y \qquad \text{ with } \alpha \in \mathbb{R}^n\,.
\end{equation}
Assuming $K_n$ to be invertible, the solution $\wh{\alpha}$ of the above equation yields
an estimator $f_{\wh{\alpha}} \in \cH$ interpolating perfectly the training data, but 
that will presumably have very poor performance in terms of 
the $\mathcal{L}^2(\nux)$ distance \eqref{eq:l2dist}, or equivalently
having poor generalization error: this is the overfitting phenomenon. 
There is a  variety of possible approaches to counteract this effect by
finding a {\em regularized} solution of \eqref{eq:kernel_ls};
perhaps the most well-known one is
\begin{equation}
\label{eq:tyk}
\wh{\alpha}_\lambda = (K_n + \lambda I)^{-1}\Y,
\end{equation}
for some fixed parameter $\lambda>0$\,, known alternatively as kernel ridge regression, Tikhonov's regularization,
least squares support vector machine, or MAP Gaussian process regression;
for a theoretical study of the convergence rate properties of this approach,
see for instance \citet{CapDe-07}, \cite{SteChr08}.

In this paper, we study the conjugate gradient (CG) technique in combination with early
stopping to determine the vector of coefficients $\wh{\alpha}$\,. 
Conjugate gradient is a computationally efficient scheme to approximatively solve linear
systems such as \eqref{eq:kernel_ls}. The principle of CG is to
restrict the problem to a nested set of data-dependent
subspaces, the so-called Krylov subspaces, defined as
\begin{equation}
\label{eq:Km}
\cK_m(\Y,K_n) := \text{vect}\set{\Y,K_n \Y,\ldots,K_n ^{m-1} \Y  } = 
\set{ p(K_n)\Y, \; p \in \cP_{m-1}}\,,
\end{equation}
where $\cP_{m-1}$ denotes the set of real polynomials of degree at most $(m-1)$\,.
Denote by $\inner{.,.}$ the usual euclidean scalar product on $\mbr^n$ rescaled by the factor $n^{-1}$\,. 
We define the $K_n$-seminorm as
$\norm{\alpha}^2_{K_n}:= \inner{\alpha,\alpha}_{K_n} 
:= \inner{ \alpha, K_n \alpha} .$
Then the CG solution after $m$ iterations is formally defined as
\begin{equation}
\label{eq:crit_cg}
\wh{\alpha}_m= \text{arg}\min_{\alpha \in \cK_m(\Y,K_n)} \|\Y -K_n \alpha\|_{K_n}\,.
\end{equation}

It is not difficult to prove from \eqref{eq:Km} and \eqref{eq:crit_cg} that 
iterating CG for $n$ iterations returns $\wh{\alpha}_n = K_n^{\dagger} \Y$\,,
where $K_n^{\dagger}$ is the pseudo-inverse of $K_n$ (in other words, the solution
to \eqref{eq:kernel_ls} if $K_n$ is invertible, and its least squares approximate
solution otherwise) and thus will suffer the same overfitting phenomenon mentioned
above. However, CG is usually stopped an early
iteration $m \ll n$\,, thus returning an approximate solution.  In
the learning context considered here, beyond computational aspects making this
method attractive, this
early stopping is mainly used for its regularization properties.
The main contribution of this paper is to study the convergence rates
of this approach when the stopping iteration $m$ is suitably chosen.

Computationally, conjugate gradients have the appealing property that the
optimization criterion \eqref{eq:crit_cg} can be computed by a simple
iterative algorithm that constructs basis vectors $d_1,\ldots,d_m$ of
$\mathcal{K}_m(\Y,K_n)\,$ by using only {\em forward multiplication}
of vectors by the matrix $K_n$\,. Algorithm \ref{algo:cg} displays the
computation of the CG kernel coefficients $\alm$ defined by
\eqref{eq:crit_cg} (see for instance \citealp{Han95}, Section~2.2 and \citealp{Engl96}, Chapter~7.)
Formally, the output of $m$ steps of the algorithm matches exactly the definition \eqref{eq:crit_cg}. In practice, finite numerical machine precision means that  rounding
errors can accumulate. Several variations of the algorithm exist, some of which
have better reported numerical stability. We will not elaborate more on this topic, since
the focus of this paper is on theoretical convergence rates.

\begin{algorithm}[H]
\caption{Kernel Conjugate Gradient regression}
\KwIn{kernel matrix $K_n$\,, response vector $\Y$\,,
maximum number of iterations $m$}
{\bf Initialization}: $\wh{\alpha}_0 ={\bf 0}_n;  r_1 = \Y; d_1 = \Y ;
t_1 = K_n \Y$\;
\For{$i=1,\ldots,m$}{
$t_{i}=t_{i}/\|t_i\|_{K_n}$\,;
$d_{i}=d_{i}/\|t_i\|_{K_n}$ (normalization of the basis, resp. update vector)\;
$\gamma_i = \inner{\Y, t_i}_{K_n}$ (proj. of $\Y$ on basis vector)\;
$\wh{\alpha}_{i} = \wh{\alpha}_{i-1} + \gamma_i d_i$ (update)\;
$r_{i+1} = r_{i} - \gamma_i t_i$ (residuals)\;
$d_{i+1} = r_{i+1}  - d_i \inner{t_i,K_nr_{i+1}}_{K_n}$\,;
$t_{i+1} = K_n d_{i+1}$
(new update, resp. basis vector)\;
}
\KwResult{CG kernel coefficients  $\alm$\,, CG function $\fm=\sum_{i=1} ^n \wh{\alpha}^{(i)}_{m} k(X_i,\cdot)$}
\label{algo:cg}
\end{algorithm}

\subsection{Relation to existing work}

As we restrict the learning problem onto the Krylov space
$\mathcal{K}_m(\Y,K_n)\,$\,, the CG coefficients $\alm$ are of the
form $\alm=q_m(K_n) \Y$ with $q_m$ a polynomial of degree $\leq
m-1$\,. However, the polynomial $q_m$ is not fixed but depends on $\Y$
as well, making the CG method nonlinear in the sense that the
coefficients $\alm$ depend on $\Y$ in a nonlinear fashion.

This is in contrast to Tikhonov's regularization \eqref{eq:tyk}, and more generally
to the larger family of {\em spectral linear regularization} methods, which
estimate the expansion coefficients via $\wh{\alpha}_\lambda = F_\lambda(K_n) \Y$\,,
where $F_\lambda$ is an appropriately regularized but fixed approximation of the inverse
function. For results on the convergence rates of such linear regularization methods
in a kernel learning setting comparable to the one studied here, see
\citet{BauPerRos07,smalezhou2,CapDe-07,LogRosetal08,Cap10} and the recent advances
\citet{DicFosHsu15,BlaMuc16}. In particular, the convergence rates under source
condition type regularity and polynomial eigenvalue decay of the kernel integral operator
obtained in the present paper for kernel CG match the rates established
for spectral linear regularization methods by \citet{CapDe-07,Cap10,DicFosHsu15}
and \cite{BlaMuc16}.

Both linear regularization methods and the nonlinear CG method are
established techniques in the {\em inverse problem} literature, in a
deterministic setting (for a comprehensive overview see
\citealp{Engl96}.) The statistical kernel learning setting is markedly
different since both the design points and the error are stochastic,
however the convergence analysis in that setting owes a lot to the
mathematical techniques developed in the deterministic case. This is
true for linear regularization methods cited above, and holds as well
for CG: the present work builds notably on the seminal works of
\citet{Han95} and \citet{Nem86}.

Conjugate gradient methods have appeared  under the name of {\em partial least squares} (PLS) in the statistics literature \citep{Wold8401}, and a ``kernelized'' version of
PLS was developed by \citet{Rosipal0101} and is now considered part of the standard
toolbox of kernel methods (see \citealp{CriSha04}, Section 6.7.2).
An important difference with the method
we study here is that kernel PLS is defined via \eqref{eq:crit_cg} but with the $K_n$-norm
replaced by the regular $n$-dimensional Euclidean norm.
In conjugate gradient parlance, kernel PLS is ``Conjugate Gradient - Minimum Error''
(CGME) while we analyze here ``Conjugate Gradient applied the the Normal Equations''
(CGNE); see \cite{Han95}, Section~2.3. Computationally, the two methods are very similar; the main reason we concentrate
on kernel CGNE rather than the perhaps more natural kernel CGME is technical:
even in the deterministic case, the analysis of CGME presents significantly more
technical difficulties (\citealp{Han95}, Chap.~4).

The results presented here are an extended version of a prior conference paper
\citep{BlaKra11}. There, a first result was obtained
directly based on \citeauthor{Nem86}'s theorem in the deterministic case: by controlling (via a
simple concentration inequality), with high
probability, the norm of the errors (on the kernel covariance as well as on the data),
it was possible to plug these deterministic estimates directly into
\citeauthor{Nem86}'s theorem, resulting in a bound holding with large probability. However, it was not possible to capture in this way
the ``fast convergence rate'' behavior
related to an assumed polynomial decay of the kernel operator's spectrum, a
phenomenon that is specific to the stochastic setting and the object of much attention
in the recent years (often under the name ``adaptation to the intrinsic data
dimensionality''). For reasons of readability, in the present version
we decided to skip this first suboptimal
(but easy to obtain via \citeauthor{Nem86}'s theorem) result, to concentrate on the
improved fast rate results. The proof of those follows closely the general structure
of \citeauthor{Nem86}'s argument and ideas, but requires a complete reworking in the details
due to the additional difficulties arising from taking into account the behavior of
the operator's spectrum. Furthermore, applying \citeauthor{Nem86}'s result also required
to assume $f^* \in \cH$\,, while the case $f^* \not\in \cH$ (called ``outer case'' below)
also introduces additional difficulties.
The present version
extends the scope of the original conference version
by also including convergence results not only in the prediction (or $\mathcal{L}^2(\nux)$) norm,
but in stronger norms as well, including the Hilbert $\cH$-norm when applicable.

\section{Mathematical framework}

\subsection{Reproducing kernel Hilbert spaces}

We assume the reader familiar with the formalism of reproducing kernel Hilbert spaces (RKHS)
and refer her for instance to \cite{CriSha04} and \cite{SteChr08} for more details. We
recall briefly a few key points. Given $k: \cX^2 \rightarrow \mbr$ a real, symmetric and
semi-definite positive kernel on $\cX$\,, the unique RKHS associated to $k$ is denoted
by $\cH$\,. We recall that $\cH$ is a Hilbert space of real-valued functions on $\cX$
containing the functions $k_x=k(x,.):= \paren{ y \in \cX \mapsto k(x,y)} $ for all $x\in \cX$
and satisfying the characteristic {\em self-reproducing} property $\inner{k_x,f}_{\cH} = f(x)$\,,
for all $f\in \cH, x\in \cX$\,. In the rest of this paper we will make the assumption that

\bigskip

{\bf (A)} $k$ is measurable, for all $x \in \cX$ it holds $k(x,x)\leq \kappa$\,, where $\kappa$ is
a real constant. The Hilbert space $\cH$ is assumed to be separable.

\bigskip

Assumption {\bf (A)} implies that the kernel integral operator
\begin{equation}
\label{eq:kernopl2}
K: \mathcal{L}^2(\nux) \rightarrow \mathcal{L}^2(\nux),\,g \mapsto \int k(.,x) g(x) dP(x)\,,
\end{equation}
is a well-defined, self-adjoint, Hilbert-Schmidt (and even trace-class) operator. 
We measure {\em regularity} of the target function $\fst$
in terms of a {\em source condition} with respect to $K$ and parameters $\rho>0,r>0$\,, 
defined as follows: 
\[
\text{\bf SC($r,\rho$)}: \text{ there exists } u \in \mathcal{L}^2(\nux) \text{ such that }
f^* = K^r u \qquad \text{ with } \qquad \norm{u}_{2,\nu} \leq \kappa^{-r}\rho.
\]
Clearly, in the above condition we can assume that $u \in \mathrm{Ker}(K)^\perp = \ol{\mathrm{Im}(K)}$ without loss of generality.
It is well-known that if $r \geq 1/2$\,,
then $f^*$ coincides almost surely with a function belonging to $\cH$\,, while
for $r< 1/2$ this is not the case. We
refer to $r\geq 1/2$ as the ``inner case'' and to $r<1/2$ as the ``outer case''.

The regularity of the kernel operator $K$ with respect to the marginal
distribution $\nux$ is measured in terms of the so-called {\bf effective dimensionality} condition.
We define the auxiliary notation $\cN(\lambda):=\tr(K(K+\lambda I)^{-1})$\,.
Given two parameters $s\in(0,1)$\,, $D\geq 1$\,, introduce the condition 
\[
\text{\bf ED($s,D$)}:  \cN(\lambda) \leq D^2 (\kappa^{-1} \lambda)^{-s}
\text{ for all } \lambda \in (0,1].
\]
This notion was first introduced by \citet{Zha05} in a learning context,
and used in a number of works since. It is related to the decay rate
of the (ordered) eigenvalues $(\xi_i)_{i\geq 1}$ of $K$\,: if those
satisfy $\xi_i \leq C i^{-1/s}$\, for some constant $C$\,, then  {\bf ED($s,D$)} is satisfied for
an appropriate constant $D$\,. On the other hand, 
under the double-sided condition $c i^{-1/s} \leq \xi_i \leq C i^{-1/s}$\,, lower bounds on
the minimax convergence rates
for the model defined by the source conditions {\bf SC($r,\rho$)} are known 
to be $\mtc{O}(n^{-2r/(2r+s)})$ for the $\cL^2(\nu)$ error \citep{CapDe-07},
resp. $\mtc{O}(n^{-(2r-1)/(2r+s)})$ for the $\cH$-norm error, assuming $r\geq 1/2$
\citep{BlaMuc16}\,.


\subsection{Conditions on the noise}

If $(X,Y) \sim P$\,, denote the noise $\eps:=Y-\e{Y|X}$\,.
We will consider -- depending on the result -- one of the following
assumptions:
\vspace{-0.2cm}
\begin{list}{}{\setlength{\leftmargin}{0cm}}
\item  {\bf (Bounded)} (Bounded $Y$): $|Y| \leq M$ almost surely.
\item {\bf (Bernstein)} (Bernstein condition): $\e{\eps^p | X } \leq (1/2) p! M^p$ almost surely, 
for all integers $p\geq 2$\,.
\end{list}
\vspace{-0.2cm}
The second assumption is weaker than the first. In particular, the
first assumption implies that not only the noise, but also the target function $f^*$ is
bounded in supremum norm, while the second assumption does not put any additional
restriction on the target function.

\section{Convergence rates}

We now introduce the early stopping rule, which takes the form of
a so-called {\bf discrepancy stopping rule:} for some 
threshold $\Omega>0$ to be specified, 
define the (data-dependent) stopping iteration $\wh{m}$ as
the first iteration $m\geq 0$ (with the convention $\wh{\alpha}_0=0$) 
for which
\begin{equation}
\norm{\Y - K_n \alm  }_{K_n} < \Omega\,.
\end{equation}
As mentioned earlier, it holds at the $n$-th iteration that $\wh{\alpha}_n = K_n^{\dagger}\Y$\,,
so that $\norm{\Y - K_n \wh{\alpha}_n  }_{K_n}=0$\,;
therefore the above stopping rule is well-defined and such that $\wh{m}\leq n$\,.
In this section, we assume that the parameters $r$ and $s$ appearing
in conditions {\bf (SC)} and {\bf (ED)} are known a priori to the user,
so that they can be used in the definition of the stopping rule.

Our first result concerns the ``inner regularity'' case ($r\geq 1/2$\,,
so that the target function $\fst$ coincides a.s. with a function belonging to $\cH$),

\begin{theorem}\label{thm:inner}
For some constant $\tau'>3/2$ and $1>\gamma>0$\,, 
consider the discrepancy stopping rule with the threshold
\begin{equation}
\label{eq:disc2}
\Omega = \tau' M \sqrt{\kappa} \paren{\frac{4D}{\sqrt{n}} \log \frac{6}{\gamma}}^{\frac{2r+1}{2r+s}}\,.
\end{equation}

Suppose that the noise fulfills the Bernstein assumption {\bf (Bernstein)},
that the source condition {\bf SC}($r,\rho$) holds for $r\geq 1/2$\,,
and that {\bf ED($s,D$)} holds. Finally, assume $n$ is large enough so that
\begin{equation}
  \label{eq:condninner}
n \geq 16 D^2 \log^2 \left(6/\gamma\right)\,.
\end{equation}

Then with probability $1-\gamma$\,, the estimator $\wh{f} = f_{\wh{\alpha}_\wh{m}}$
obtained by the discrepancy stopping rule \eqref{eq:disc2}
satisfies for any $\theta\in [0,\frac{1}{2}]$\,:
\[
\norm{K^{-\theta}(\wh{f} - f^*)}_{2,\nux} \leq
c(r,\tau)
(\rho+M) \kappa^{-\theta} \paren{\frac{4D}{\sqrt{n}}\log \frac{6}{\gamma}}^{\frac{2(r-\theta)}{2r+s}}\,.
\]
\end{theorem}

Observe that in the above result, taking $\theta=\frac{1}{2}$ results in
a convergence rate result in the $\cH$-norm: this is because for a function
$g$ that coincides a.s. with a function $g_{\cH}\in \cH$, it holds
$ \big\|K^{-\frac{1}{2}}g\big\|_{2,\nux} = \norm{g_{\cH}}_{\cH}\,$ (see Section~\ref{eq:proofsetup} for details). Thus
we obtain (simultaneous) convergence rate results for the $\mathcal{L}^2(\nux)$-norm,
the $\cH$-norm as well as all intermediate norms.

\newcommand{\wtn}{\wt{n}}

We now turn to the ``outer rate'' case ($r<\frac{1}{2}$). 
In this situation, following an idea used by \cite{Cap10},
we make the additional assumption that
{\em unlabeled} data is available. Assume that we have $\wtn$ i.i.d. observations $X_1,\ldots,X_{\tilde n}$\,, out of which only the first $n$ are labeled. We define a new response vector $\wY=\frac{\tilde n}{n} \left(Y_1,\ldots,Y_n,0,\ldots,0\right)\in \mathbb{R}^{\tilde n}$ and run the CG algorithm \ref{algo:cg} on $X_1,\ldots,X_{\tilde n}$ and $\wY$\,. We use 
the stopping rule with the following threshold:
\begin{equation}
  \label{eq:outthr}
  \Omega =  \tau' \max(\rho,M) \sqrt{\kappa} \paren{\frac{4D}{\sqrt{n}} \log \frac{6}{\gamma}}^{\frac{2r+1}{2r+s}}\,.
\end{equation}
Observe that it is similar to \eqref{eq:disc2} in the previous section, except the factor
$M$ is replaced by $\max(M,\rho)$ (and the numerical constants are different).
\begin{theorem}\label{thm:outer}
  For some constant $\tau'>6$\,, and $\gamma \in (0,1)$\,,
  consider the discrepancy stopping rule with the fixed threshold given by \eqref{eq:outthr}\,.

Suppose assumptions {\bf (Bounded)},
{\bf SC($r,\rho$)} and  {\bf ED($s,D$)}, are granted with $r<\frac{1}{2}$ and
$r+s \geq \frac{1}{2}$\,. Assume $n$ is large enough so that
\begin{equation}
  \label{eq:condnouter}
n \geq 16 D^2 \log^2 \left(4/\gamma\right)\,,
\end{equation}
and that additional unlabeled data is available with $\wt{n} \geq n^{\frac{1+s}{2r+s}}\,.$
Then with probability $1-\gamma$\,,the estimator $\wh{f}$
obtained by the discrepancy stopping rule defined above satisfies
for any $\theta\in[0,r)$\,:
\[
\norm{K^{-\theta}(\wh{f}-f^*)}_{2,\nu} \leq 
c(r,\tau)
(\rho+M) \kappa^{-\theta} 
\paren{\frac{4D}{\sqrt{n}}\log\frac{4}{\gamma}}^{\frac{2(r-\theta)}{2r+s}}\,.
\]
\end{theorem}

In the outer case, since $f^* \not\in \cH$ we of course cannot expect any convergence
in $\cH$-norm, but as is clear from the above result, we obtain convergence
rate results in norms that are stronger than the $\mathcal{L}^2(\nu)$-norm, with the meaningful range
of $\theta$ (determining the strength of the norm) determined by the source condition
parameter $r$\,.

%

\section{Discussion}

{\em Rate quasi-optimality.}  Convergence rates are generally stated in expectation,
while the convergence results of Theorems~\ref{thm:inner} and~\ref{thm:outer} are stated
with high probability. Usually, exponential deviation bounds can be integrated
to yield bounds in expectation; unfortunately, this is not directly possible here,
because (a) conditions \eqref{eq:condninner}
(resp. \eqref{eq:condnouter}) introduce a constraint between the number of examples $n$
and the probability $\gamma$ that the bound fails, preventing a statement about
extreme (exponentially small in $n$) quantiles of the error; and even 
more importantly (b) the threshold $\Omega$
for the stopping criterion itself depends on the prescribed bound failure probability $\gamma$\,.
For the present discussion, we therefore consider the following slightly weaker notion considered by \citet{CapDe-07}: we call a positive sequence $(a_{n,\theta})_{n\geq 1}$ an upper rate of convergence 
in probability for the sequence of 
estimators $\wh{f}_n$ over a class of distributions $\mtc{P}$ if
\begin{equation}
\label{eq:ratesprob}
\lim_{C \rightarrow \infty} \limsup_{n\rightarrow \infty} \sup_{P \in \cP}
\probb{(X_i,Y_i)_{i=1}^n \stackrel{\mathrm i.i.d.}{\sim} P}
{\norm{K^{-\theta}(\wh{f}_n-f_{P}^*)}_{2,\nu} > C a_{n,\theta}} =0\,,
\end{equation}
On the other hand, for the class of distributions defined by the source condition {\bf SC($r,\rho$)}
and the  polynomial decay condition
$c i^{-1/s} \leq \xi_i \leq C i^{-1/s}$\, for the eigenvalues $(\xi_i)_{i\geq 1}$ of the kernel integral
operator $K$,
the convergence rate $a^*_{n,\theta} := n^{-\frac{r-\theta}{2r+s}}$
is minimax optimal 
(for lower bounds on attainable rates see, for instance, \citealp{CapDe-07} for the case $\theta = 0$\, and
\citealp{BlaMuc16} for the case $r\geq \frac{1}{2}, \theta \in [0,\frac{1}{2}]$\,.)

We can thus conclude that kernel CG enjoys quasi-optimal statistical rates of
convergence, in the sense that any sequence  $a_{n,\theta}$
such that $a^*_{n,\theta}=o(a_{n,\theta})$ can be an upper rate of convergence, 
provided the stopping iteration is chosen appropriately.
Namely, we can choose the sequence of bound failure probabilities $(\gamma_n)_{n\geq 1}$
converging to $0$ arbitrarily slowly, so that the rate obtained using the corresponding sequence $\Omega_n$
in Theorems~\ref{thm:inner} or~\ref{thm:outer} yield \eqref{eq:ratesprob}
with a rate $a_{n,\theta} = a^*_{n,\theta} c_n$ where $c_n$ tends arbitrarily slowly to $\infty$.


{\em Comparison of methods.} It is known from previous works that
a large family of spectral linearization methods (see in particular
\citealp{CapDe-07} for Tikhonov regularization, and \citealp{Cap10,BlaMuc16} for
the general case) achieve minimax optimal convergence rates in the setting
considered here. It is therefore a legitimate question whether the additional
technicality for analyzing CG is justified, given the avaibility of other methods.
The main reason is that CG remains an algorithm of choice because of its excellent
computational properties. Because it very agressively aims at reducing
the residual error, it is often observed in practice that CG converges in much fewer
iterations than other methods (such as regular gradient descent, studied in
the kernel learning setting by \citealp{YaoRosCap07}). For this reason it would
be of interest to analyze the convergence rate of kernel PLS as well, which,
as already mentioned in the introduction, is computationally very similar to kernel CG
but more challenging to study theoretically.

{\em Adaptivity.} The quasi-optimal convergence rates obtained in this work use a stopping
criterion with a threshold depending on the parameters $r$ and $s$\, (and on $\rho$
in the outer case). It is unrealistic to assume these various regularity parameters
to be known in advance in practice. The question of automatic choice of a stopping
iteration without prior knowledge of these parameters is known as that of
adaptivity. For what concerns the convergence in excess prediction error
$\mathcal{L}^2(\nux)$\,, and under the assumption {\bf (Bounded)}: $|Y|\leq M$\,,
it is well known that a simple hold-out strategy (i.e. choosing, amongst
a family of candidate estimators,
the one achieving minimal error on an held-out validation sample),
performed after trimming all candidate estimators $\wh{f}_m$ to the interval $[-M,M]$\,,
generally speaking selects an estimator close to the best between those considered.
In the present context, each iteration $m$ provides such a candidate estimate; 
one could for example adapt the corresponding arguments from \citet{Cap10}\,;
see also \cite{BlaMas06} for a general point of view on this question.
As a consequence, since the results established here grant the {\em existence} of
an iteration with quasi-optimal rate, this will also be the case for the
adaptive hold-out strategy. It remains an open question whether this also applies
to the error measured in stronger norms: although we established that there
exists an iteration with optimal rates for all norms, it does not follow that
any iteration which is good in the sense of excess prediction (one of which
hold-out would select) would automatically also yield good performance for the
stronger norms.

\section{Proofs}

The proof of our results rely on combining ideas of \cite{Han95} 
(expanding upon the remarkable seminal work of \citealp{Nem86}), 
used in the analysis of convergence of CG algorithms for (deterministic) inverse problems,
with tools introduced by \cite{Vitetal06},
\cite{CapDe-07}, \cite{smalezhou2}, \cite{Cap10} for the analysis of inverse problem methods
for the statistical learning setup. We start by gathering in the two next sections
the required notation and previous results that we will make use of, before getting to the proof
itself. In all the proofs, we  use the notation $c(a,b)$ to denote a nonrandom function only depending
on the nonrandom parameters $a,b$\,, and whose exact value can change from line to line.

\subsection{Setup and key tools for statistical learning as an inverse problem}
\label{eq:proofsetup}

We first define the {\em empirical evaluation operator} $T_n$ as follows:
\[
T_n: \qquad g \in \cH \mapsto T_ng :=(g(X_1),\ldots,g(X_n))^\top \in \mbr^n
\]
and the {\em empirical integral operator} $T_n^*$ as:
\[
T_n^*:  u=(u_1,\ldots,u_n) \in \mbr^n \mapsto
T_n^*u := \frac{1}{n} \sum_{i=1}^n  u_i k(X_i,\cdot) \in \cH.
\]
Using the reproducing property of the kernel, it can be readily
checked that $T_n$ and $T_n^*$ are adjoint operators, i.e. they satisfy
$\inner{T^*_n u,g}_{\cH} = \inner{u,T_ng}$\,, for all $u \in \mbr^n, g\in \cH$\,. 
With this notation, it is clear that if $\wh{\alpha} \in \mbr^n$ is the
vector of coefficients in the normalized kernel expansion \eqref{eq:kernexp}
of a kernel estimator $\wh{f}$\,, then if holds $\wh{f} = T_n^* \wh{\alpha}$\,.
Furthermore,  since $K_n=T_nT_n^*$\,, we have for any $u\in\mbr^n$\,:
\[
\norm{u}^2_{K_n} = \inner{u,K_n u} = \norm{T_n^* u}^2_{\cH}\,.
\]
Based on these facts, equation \eqref{eq:crit_cg} can be rewritten as
\[
\alm= \text{arg}\min_{\alpha \in \cK_m(\Y,K_n)} \| T_n^* \Y - T_n^* T_n T_n^* \alpha\|_{\cH}\,,
\]
implying that for the $m$-th iteration estimator $\fm = T_n^* \alm$\,, it holds
\begin{equation}
\label{eq:crit_cg2}
f_{m}= \text{arg}\min_{f \in \cK_m(T_n^*\Y,S_n)} \| T_n^*\Y - S_n f\|_{\cH}\,,
\end{equation}
where $S_n=T_n^*T_n$ is a self-adjoint operator of $\cH$\,, called empirical covariance
operator. In the sequel we will mainly refer to \eqref{eq:crit_cg2} as the characterization of
the CG method.
In fact, \eqref{eq:crit_cg2} corresponds to the definition of the ``usual'' conjugate 
gradient algorithm (in Hilbert space), formally applied to the so-called normal equation (in $\cH$)
\begin{equation}
\label{eq:normaleq}
S_n \wh{f} = T_n^* \Y\,,
\end{equation}
which is obtained from \eqref{eq:kernel_ls} by left multiplication by $T_n^*$\,. 

The advantage of this reformulation, and an idea first introduced by
\cite{Vitetal06}, is that it can be interpreted as a perturbation
of a {\em population, noiseless} version (of the equation and of the algorithm), wherein
$\Y$ is replaced by the target function $\fst$ and the empirical operators $T_n^*,T_n$ are
respectively replaced by their population analogues, the kernel integral operator
\[
T^* : g \in \mathcal{L}^2(\nux) \mapsto T^*g := \int k(x,.) g(x) d\nux(x) = \e{k(X,\cdot) g(X)} \in \cH\,,
\]
and the change-of-space (or inclusion) operator
\[
T: \qquad g \in \cH \mapsto g \in \mathcal{L}^2(\nux)\,.
\]
The latter maps a function to itself
but between two Hilbert spaces which differ with respect to their geometry -- the inner
product of $\cH$ being defined by the kernel function $k$\,, while the
inner product of $\mathcal{L}^2(\nux)$ depends on the data generating
distribution. This operator is well defined:
since the kernel is bounded, all functions in $\cH$
are bounded and therefore square integrable under any distribution $\nux$\,; this
also implies that $T^*$ is well-defined. Again,
it can be checked due to the reproducing property that $T,T^*$ are adjoint of each other;
we denote $S:= T^* T$ the population covariance operator, and observe that $K=TT^*$ holds,
where $K$ is the operator defined by \eqref{eq:kernopl2}. Finally, it holds that
$S^{-\frac{1}{2}}T^*$ is a partial isometry from $\mathcal{L}^2(\nux)$ to $\cH$\,,
and $K^{-\frac{1}{2}}T$ a partial isometry from $\cH$ to $\mathcal{L}^2(\nux)$\,.
In particular, if $g\in \mathcal{L}^2(\nux)$ coincides a.s. with some function $g_\cH \in \cH$\,,
then it holds $g=Tg_\cH$ with
$\norm{g_{\cH}}_\cH = \big\|K^{-\frac{1}{2}}Tg_{\cH}\big\|_{2,\nux} = \big\|K^{-\frac{1}{2}}g\big\|_{2,\nux}$\,.

The next lemma was established by \citet{CapDe-07}, based on a
Bernstein-type inequality for random variables taking values in a Hilbert space,
see \cite{PinSak85,Yur95}. It bounds with high probability the deviations between the 
quantities in the normal equations \eqref{eq:normaleq} and their population counterparts.
A key insight from \citet{CapDe-07} is that in order to obtain sharp bounds on convergence
rates, these deviations should be measured in a ``warped'' norm rather than in the standard norm:
\begin{lemma}
\label{prop:relconc}
Let $\lambda$ be a positive number. Under assumption {\bf (Bounded)}, the following holds:
\begin{equation}
\label{eq:byby}
\prob{\norm{(S+\lambda I)^{-\frac{1}{2}}(T_n^* \Y - T^*f^*)}_{\cH} \leq
2M\paren{\sqrt{\frac{\cN(\lambda)}{n}} + \frac{2\sqrt{\kappa}}{\sqrt{\lambda}n}}
\log \frac{2}{\gamma} }
\geq 1-\gamma\,.
\end{equation}
If the representation $f^*=Tf^*_{\cH}$ holds and under assumption {\bf (Bernstein)},
we have the following:
\begin{equation}
\label{eq:bybn}
\prob{\norm{(S+\lambda I)^{-\frac{1}{2}}(T_n^* \Y - S_n f^*_{\cH})}_{\cH} \leq
2M\paren{\sqrt{\frac{\cN(\lambda)}{n}} + \frac{2\sqrt{\kappa}}{\sqrt{\lambda}n}}
\log \frac{2}{\gamma}}
\geq 1-\gamma\,.
\end{equation}
Concerning the convergence of empirical covariance in the sense of operators, the following holds:
\begin{equation}
\label{eq:reloperror}
\prob{\norm{(S+\lambda I)^{-\frac{1}{2}}(S_n-S)}_{HS} \leq
2\sqrt{\kappa} \paren{ \sqrt{\frac{\cN(\lambda)}{ n}} + \frac{2\sqrt{\kappa}}{\sqrt{\lambda} n}}
\log \frac{2}{\gamma} }
\geq 1-\gamma\,,
\end{equation}
as well as:
\begin{equation}
\label{eq:hoeffop}
\prob{\norm{S_n-S}_{HS} \leq
\frac{4\kappa}{\sqrt{n}}\sqrt{\log \frac{2}{\gamma}} }
\geq 1-\gamma\,,
\end{equation}
where $\norm{.}_{HS}$ denotes the Hilbert-Schmidt norm.
\end{lemma}

\subsection{Key tools for the analysis of CG using orthogonal polynomials theory}

We denote by $(\xini,\eni)_{i\in I}$ (respectively  $(\xi_{i},e_{i})_{i\in I}$)
an eigenvalue-eigenvector orthogonal basis with $\xini,\xi_{i} \in [0,\kappa]$ 
for the operator $S_n$\,, respectively $S$\,. 
Since the rank of $S_n$ is at most $n$\,, the family $(\xini)_{i \in I}$ has at most $n$ nonzero terms
(while the family $(\xi_i)_{i \in I}$ has at most countably many nonzero terms, $S$ being compact).

Using the formalism of functional calculus for operators, if $\phi:[0, \kappa] \rightarrow \mbr$ is
a bounded and measurable function, we denote
\[
\phi(S_n) := \sum_{i\in I} \phi(\xini) \eni e^*_{n,i} \;\; \text{ and }  \;\;
\phi(S) := \sum_{i\in I} \phi(\xi_{i}) e_{i} e^*_{i}\,.
\]
We recall the ``switching'' rule $T\phi(S)=\phi(K)T^*$ since $S=T^*T$\,, $K=TT^*$\,,
which we will be using often.
In the sequel, $\norm{A}$ denotes the operator norm of an operator $A$\,. The above definition implies in particular the bound
\[
\norm{\phi(S_n)} \leq \sup_{t \in [0,\kappa]} \abs{\phi(t)}.
\]
For $u\geq 0$\,, denote $F_{u} = \mbf{1}_{[0,u)}(S_n)$\,, that is,
the orthogonal projector in $\cH$ onto the subspace 
$\mathrm{vect}\set{\eni , i\in I \;\mathrm{ s.t. }\; \xini < u}$ spanned by eigenvectors 
of $S_n$ corresponding to eigenvalues strictly less than $u$\,. Observe that
$F_u$ depends on the data because $S_n$ does, but we omit the index $n$ at this juncture
to simplify notation, and because we will not make use of the corresponding 
notion for $S$\,, so that there is no risk of confusion. Finally, 
for an integer $\ell$ introduce  the
measure
\[
\mu_n^{(\ell)} := \sum_{i\in I} \xini^\ell \inner{T_n^* \Y,\eni}^2 \delta_{\xini},
\]
where $\delta_x$ denotes the Dirac delta-measure at point $x$\,. In particular,
for $\ell=0$ we use the convention $0^0=1$ and it holds for a bounded measurable function 
$\phi: [0, \kappa] \rightarrow \mbr$\,:
\[
\norm{\phi(S_n)T_n^* \Y}^2_{\cH} = \sum_{i\in I} \phi(\xini)^2 \inner{T_n^* \Y,\eni}^2 = 
\int \phi(t)^2 d\mu_n^{(0)}(t)\,.
\]
Observe that, for $i\in I$ such that $\xini=0$\,, we have $\inner{T_n^*\Y,\eni}=\inner{\Y,T_n \eni}=0$
since $\eni \in \mathrm{ker}(T_n^* T_n) = \mathrm{ker}(T_n)$\,. Therefore, the measures $\mu_n^{(\ell)}$
have finite support (independent of $\ell$) 
of cardinality $n_\Y \leq n$\,. In fact $n_\Y$ is the number of distinct positive eigenvalues
of $K_n$ such that $\Y$ has nonzero projection on the corresponding eigenspace
(or equivalently, the number of distinct positive eigenvalues
of $S_n$ such that $T_n^*\Y$ has nonzero projection on the corresponding eigenspace).

With this formalism established, we turn to properties of the CG method
(see, e.g., \citealp{Han95}, and \citealp{Engl96}, Chapter~7).
By its definition, the output of the $m$-th iteration of the CG algorithm
can be put under the form $\fm = q_m(S_n) T_n^* \Y\,,$ where $q_m \in \cP_{m-1}$\,, 
the vector space of real polynomials of degree at most $m-1$\,.
A crucial role is played by the {\em residual polynomial}
\[
p_m(x) = 1 - xq_m(x) \in \cP_{m}^0\,,
\]
where $\cP_{m}^0$ is the affine space of real polynomials of degree no
greater than
$m$ and having constant term equal to~1. In particular
$T_n^*\Y - S_n\fm = p_m(S_n)T_n^*\Y$\,. For $\ell\geq 0$ we define
\begin{align*}
[p,q]_{(\ell)}  := \int_0^{\kappa} p(t)q(t)  d\mu_n^{(\ell)}(t) 
 & = \inner{p(S_n)T_n^*\Y,S_n^\ell q(S_n)T_n^*\Y}\\&  = \sum_{i\geq 1} p(\xini)q(\xini)
\xini^\ell 
\inner{T_n^*\Y,\eni}^2\,.
\end{align*}
Since the measure $\mu_n^{(\ell)}$ has support of cardinality $n_\Y$\,, 
$[.,.]_{(\ell)}$ is a scalar product on the space $\cP_{n_\Y-1}$\,.
Consider an iteration $m < n_\Y$\,. By \eqref{eq:crit_cg2}, $q_m$
is the minimizer of $\norm{(I - S_n q(S_n))T_n^*\Y}^2_\cH$
over $q\in\cP_{m-1}$\,, so that the residual polynomial $p_m$ is equivalently
a minimizer of $\norm{p(S_n)T_n^*\Y}^2_\cH = [p,p]_{(0)}$ over $p \in \cP^0_m $\,.
In other words, $p_m$ is the orthogonal projection  of the origin onto the affine
subspace $\cP^0_m\subset \cP_m$ for the scalar product $[.,.]_{(0)}$\,. This, in passing, 
shows the unicity of $p_m$\,, and by consequence of $q_m$ and $f_m$\,.
In the case $m=0$\,, we set $q_0=0, p_0\equiv 1$\,.

We denote by $\pi$ the shift operation on polynomials with $(\pi q)(x)=xq(x)$\,.
Since $\cP^0_m = 1 + \pi \cP_{m-1}$ is an affine subspace of $\cP_m$ parallel to $\pi\cP_{m-1}$\,,
it follows by the properties of projections that $p_m$ is orthogonal
to $\pi\cP_{m-1}$ for $[.,.]_{(0)}$\,. Thus $0=[p_m,\pi q]_{(0)}=[p_m,q]_{(1)}$ for any $q \in
  \cP_{m-1}$\,; this establishes that $p_0,p_1,\ldots,p_{n_\Y-1}$ is an orthogonal polynomial
sequence with respect to $[.,.]_{(1)}$\,. For $m=n_\Y$\,, this is somewhat of a special case 
since $[.,.]_{(\ell)}$
is only a semidefinite product on $\cP_{n_\Y}$\,. However, it is not difficult to see that
the polynomial $p_{n_\Y}$ having $n_\Y$ distinct roots corresponding to the atoms
of $\mu_n^{(0)}$ and normalized to have constant term equal to 1, is the unique element
of $\cP^0_{n_\Y}$ satisfying $[p_{n_\Y},p_{n_\Y}]_{(0)}=0$\,. Therefore, unicity of the solution
also holds for $m=n_{\Y}$\,, and obviously also $[p_{n_\Y},p_m]_{(1)}=0$ for all $m\leq n_\Y$\,.
The CG method will not in any case go beyond iteration $m=n_\Y$\,, since at this point,
by the above considerations the residual norm is 0 and an exact solution to the equation
$S_n f = T^*_n \Y$ has been reached.

The next lemma gathers the technical results coming from the theory of
orthogonal polynomials needed for our analysis.
\begin{lemma}
\label{lem:orthpol}
Let $m$ be any integer satisfying $1\leq m \leq n_\Y$\,.

i) The polynomial $p_m$ has exactly $m$ distinct 
roots belonging to $(0,\kappa]$\,, denoted by $(x_{k,m})_{1 \leq k \leq m}$
in increasing order. 

ii) $p_m$ is positive, decreasing and convex on the interval $[0, x_{1,m})$\,.

iii) Define the function $\varphi_m$ on the interval $[0,x_{1,m})$ as
\[
\varphi_m(x) = p_m(x) \paren{\frac{x_{1,m}}{x_{1,m}-x}}^{\frac{1}{2}}\,.
\]
Then it holds
\begin{equation}
\label{eq:Nem1}
[p_m,p_m]_{(0)}^{\frac{1}{2}} = \norm{p_m(S_n)T_n^*\Y}_{\cH}  \leq \norm{F_{x_{1,m}} \varphi_m(S_n) T_n^* \Y }_{\cH}\,,
\end{equation}
and furthermore, for any $\nu\geq 0$ (and the convention $0^0=1$):
\begin{equation}
\label{eq:boundphi}
\sup_{x \in [0,x_{1,m}]} x^\nu \varphi^2_m(x) \leq \nu^\nu \abs{p'_m(0)}^{-\nu}\,.
\end{equation}

iv) Denote $p_0^{(2)},p_1^{(2)}, \ldots, p_{n_\Y-1}^{(2)}$ the unique sequence of orthogonal
polynomials with respect to $[.,.]_{(2)}$ and with constant term equal to 1. This sequence enjoys
properties (i) and (ii) above, with $(x^{(2)}_{k,m})_{1 \leq k \leq m}$ denoting the
distinct roots of $p_m^{(2)}$ in increasing order.
Then it holds  that $x_{1,m} \leq x^{(2)}_{1,m}$\,. Finally, the following holds (Christoffel-Darboux
identity):
\begin{equation}
\label{eq:pcompar}
0 \leq {{p'_{m-1}}(0)} - {{p'_{m}}(0)} =
\frac{\brac{p_{m-1},p_{m-1}}_{(0)} -
  \brac{p_{m},p_{m}}_{(0)} }{\brac{p^{(2)}_{m-1},p^{(2)}_{m-1}}_{(1)}}
\leq \frac{\brac{p_{m-1},p_{m-1}}_{(0)}}{\brac{p^{(2)}_{m-1},p^{(2)}_{m-1}}_{(1)}}\,.
\end{equation}
\end{lemma}
For a proof of these properties see the monograph of \citet{Han95},
from which the above properties have been collected.
Existence of a unique family of orthogonal polynomials $p_k^{(\ell)}$ for any $\ell\geq 0$\,, up
to degree $ n_\Y-1$, is guaranteed by the fact that
the measures $\mu_n^{(\ell)}$ have support of cardinality $n_\Y$\,.
Point (i) is well-known in the theory of orthogonal polynomials,
see also \citet{Han95}, Section~2.4.
Point (ii) is equally well-known and an easy consequence
of (i), namely (ii) holds true for any real polynomial of degree $m$ having $m$
real positive roots and taking a positive value at 0, due to the interlacing property of the roots of its derivatives. For point (iv), all roots of $p_m^{(2)}$ are positive by standard results of
orthogonal polynomial theory, so
we can normalize these polynomials to have constant term equal to 1.
Relation \eqref{eq:pcompar}, resp. the relation $x_{1,m} \leq x_{1,m}^{(2)}$ 
can be found as Corollary 2.6, resp. 2.7, of \citet{Han95}.
 Finally, point (iii) can be found
as a ingredient of the proof of Lemma~3.7 of \citet{Han95}, more
precisely \eqref{eq:Nem1},\eqref{eq:boundphi} are found respectively
as  (3.8) and (3.10) there (or equivalently as (7.7) and (7.8) in Chapter~7
of \citealp{Engl96}).
The seminal idea of introducing the function $\varphi_m$ above and properties
\eqref{eq:Nem1}-\eqref{eq:boundphi} are originally due to \citet{Nem86}.

\subsection{Proof of Theorem~\ref{thm:inner}}
\label{se:prinner}

We recall that since we assume $r\geq 1/2$\,, there exists $\fst_\cH \in \cH$ such
that $f^* = Tf^*_\cH$ holds. The main effort below is to analyze the algorithm
when the events of high probability of Lemma \ref{prop:relconc} are satisfied.
To simplify notation, we will define the following event, where
$\Lambda \geq 1, \Delta \geq 0$ are constants and $\delta(\lambda)\geq 0$
only depends on $\lambda$\,:
\[
\text{ \bf B}(\lambda) :
\begin{cases}
\norm{(S+\lambda I)^{-\frac{1}{2}}(T_n^*\Y-S_nf^*_\cH)}_{\cH} & \leq \delta(\lambda)\,,\\
\displaystyle \norm{(S+\lambda I)(S_n
  +\lambda I)^{-1}}_{HS} & \leq \Lambda^2,\\
\norm{S-S_n}_{HS} &\leq \kappa\Delta\,.
\end{cases}
\]
In the rest of this proof we set $\mu=r-1/2$\,. 
Under the source condition assumption {\bf SC($r,\rho$)}, for
$r\geq\frac{1}{2}$ the
representation $f^* = K^r u$ can be rewritten
\[
f^* = (TT^*)^ru = T (T^*T)^{r-\frac{1}{2}} (T^*T)^{-\frac{1}{2}} T^*u
= T S^\mu S^{-\frac{1}{2}} T^*u,
\]
by identification we therefore have the source condition for $f_{\cH}$
given by $f_{\cH} = S^\mu w$ with 
$w =S^{-\frac{1}{2}} T^*u $\,, and $\norm{w}_{\cH} \leq \norm{u}_{2,\nux} \leq \kappa^{-\mu - \frac{1}{2}} 
\rho$\,, 
since $S^{-\frac{1}{2}} T^*$ is a partial isometry from $\mathcal{L}^2(\nux)$ into $\cH$\,.

Finally we define following shortcut notation for $\beta>0$\,:
\begin{equation}
\label{eq:defZ}
Z_\beta(\lambda) =
\begin{cases}
\lambda^\beta & \text{ for } \beta \leq 1\,, \\
\kappa^\beta \Delta & \text{ for } \beta > 1.\\
\end{cases}
\end{equation}
In order to simplify notation, for the remainder of the paper we will
omit the indices from $\norm{.}_\cH$ and $\norm{.}_{2,\nux}$\,;
in other words the notation $\norm{.}$ will be overloaded to mean operator,
$\cH$ or $\mathcal{L}^2(\nux)$ norm, depending (nonambiguously) on the context.
Note that the $\mathcal{L}^2(\nux)$ norm will not be explicitly used again
until the proof
of Theorem~\ref{thm:outer} in Section~\ref{proof:outer}.

We start with a technical lemma encapsulating a couple of bounding devices 
that we will use repeatedly.
\begin{lemma}
\label{le:lemma0}
Let $\lambda>0$ be fixed.  
Assume the event {\bf B($\lambda$)} is satisfied.
For any $\nu \in [0,1]$\,, it holds
\begin{equation}
\label{eq:boundin0}
\norm{(S +\lambda I)^{\nu}(S_n +\lambda I)^{-\nu}} \leq \Lambda^{2\nu}\,.
\end{equation}
For any $\nu \in [0,1]$\,, and any $h \in \cH$\,, it holds
\begin{equation}
\label{eq:boundin1}
\norm{S^{\nu} h} \leq \Lambda^{2\nu} \norm{(S_n + \lambda)^{\nu} h}\,.
\end{equation}
For any $\nu>0$ and for any $\phi: [0,\kappa] \rightarrow \mbr$
measurable function, it holds
\begin{equation}
\label{eq:boundin2}
\norm{\phi(S_n)S^\nu} \leq \Lambda^{2} \paren{\sup_{t\in[0,\kappa]} t^\nu \phi(t)
+ (\nu \vee 1) Z_\nu(\lambda) \sup_{t\in[0,\kappa]} \phi(t)}\,.
\end{equation}
\end{lemma}
\begin{proof}
Inequality \eqref{eq:boundin0} is a direct consequence of the second component
in event {\bf B($\lambda$)}, and of the operator norm inequality $\norm{A^\nu B^\nu} \leq
\norm{AB}^\nu$ for self-adjoint positive operators.  See \citet{Bat97}, Theorem X.1.1,
where the result is stated for positive matrices,  but the proof applies as well to positive operators on a Hilbert space.
For the second inequality, we have
\[
\norm{S^{\nu} h} \leq \norm{S^{\nu}(S +\lambda I)^{-\nu}} 
\norm{(S +\lambda I)^{\nu}(S_n +\lambda I)^{-\nu}} \norm{(S_n +\lambda I)^{\nu} h}
\leq \Lambda^{2\nu} \norm{(S_n + \lambda)^{\nu} h}\,.
\]
Concerning the last part of the lemma, we first consider the case $\nu>1$\,; then
\begin{align*}
\norm{\phi(S_n)S^\nu} & \leq
\norm{\phi(S_n)}\norm{(S^\nu-S_n^\nu)} + \norm{\phi(S_n)S_n^\nu} \\
& \leq  \sup_{t\in[0,\kappa]} t^\nu \phi(t) + \norm{(S^\nu-S_n^\nu)} \sup_{t\in[0,\kappa]} \phi(t)
\end{align*}
Furthermore, we have
\[
\norm{(S^\nu-S_n^\nu)} \leq \norm{(S^\nu-S_n^\nu)}_{HS} \leq 
\nu \kappa^{\nu-1} \norm{S-S_n}_{HS} \leq
\nu \kappa^\nu \Delta.
\]
The second inequality used that if $A,B$ are two semipositive 
self-adjoint Hilbert-Schmidt operators,
and $\phi$ is a $L$-Lipschitz function on $[0,\max(\norm{A},\norm{B})]$\,, 
then $\norm{\phi(A)-\phi(B)}_{HS}\leq L \norm{A-B}_{HS}$
(see, for instance, \citealp{Bat97}, Lemma VII.5.5, for a proof in the finite dimensional
case that can be easily extended to the Hilbert-Schmidt case. Note 
in passing that this inequality does not hold for the operator norm in general). 
We applied this property to the power function $x\mapsto x^\nu$\,, which is 
$\nu \kappa^{\nu-1}$-Lipschitz over $[0, \kappa]$\,.

In the case $\nu \leq 1$\,, we have
\begin{align*}
\norm{\phi(S_n)S^\nu} & \leq
\norm{\phi(S_n)(S_n+\lambda I)^\nu}
\norm{(S_n+\lambda I)^{-\nu}(S+\lambda I)^\nu}
\norm{(S+\lambda I)^{-\nu}S^\nu}\\
& \leq  \Lambda^{2\nu} \paren{\sup_{t\in[0,\kappa]} t^\nu \phi(t) + \lambda^{\nu} 
\sup_{t\in[0,\kappa]} \phi(t)}\,.
\end{align*}
\end{proof}

\begin{lemma}[Bounding the error]
\label{le:lemma2}
Assume condition  {\bf SC($r,\rho$)} holds, $r\geq \frac{1}{2}$\,.
For any $\lambda>0$\,, if the event {\bf B($\lambda$)} is satisfied,
then for any iteration step $1 \leq m \leq n_{\Y}$ and
$\theta\in[0,\frac{1}{2}]$\,, for any
$\eps \in (0,x_{1,m})$\,, and denoting $\wt{\eps}:=\min(\eps,\abs{p'_m(0)}^{-1})$\,:
\begin{align*}
\norm{S^{\frac{1}{2}-\theta}(f_m-f^*_\cH)} \leq & c(\Lambda,\mu) \Bigg(
{\wt{\eps}}^{-1} \paren{ \wt{\eps} + \lambda}^{1-\theta}
 \delta(\lambda)
+\paren{\eps^\mu + Z_\mu(\lambda)}(\eps+\lambda)^{\frac{1}{2}-\theta}
\kappa^{-\mu-\frac{1}{2}} \rho \\
& \qquad \qquad +
\eps^{-1} \paren{\eps + \lambda}^{\frac{1}{2}-\theta} 
\norm{T_n^*(T_nf_m-\Y)} \Bigg)
\end{align*}
For $m=0$\,, the above inequality is valid for any $\eps>0$\,.
\end{lemma}
\begin{proof}
Set $\bar{f}_m=q_m(S_n)S_nf^*_\cH$\,. This is the element  in $\mathcal{H}$ that we obtain
by applying  the $m$th-iteration CG polynomial $q_m$ to the {\em noiseless} data.  We have
using \eqref{eq:boundin1}
\begin{align*}
\norm{S^{\frac{1}{2}-\theta}(f_m-f^*_\cH)} & \leq \Lambda^{1-2\theta} \norm{(S_n+\lambda I)^{\frac{1}{2}-\theta}(f_m - f^*_\cH)}\\
&\leq
\Lambda\left(\norm{F_{\eps}(S_n+\lambda I)^{\frac{1}{2}-\theta}(f_m-\bar{f}_m)}
+ \norm{F_{\eps} (S_n+\lambda I)^{\frac{1}{2}-\theta}(\bar{f}_m-f^*_\cH)}\right. \\
& \qquad \; \; \left. + \norm{F_{\eps}^\perp (S_n+\lambda I)^{\frac{1}{2}-\theta}(f_m-f^*_\cH)} \right)\\
  &:= \Lambda ( (I) + (II) + (III) )\,,
\end{align*}
where we denote $F_\eps^\perp := (I-F_\eps)$\,.  We upper bound the first summand
and start with using the first component of event {\bf B}($\lambda$):
\begin{align*}
(I) = \norm{F_{\eps}(S_n+\lambda I)^{\frac{1}{2}-\theta}(f_m-\bar{f}_m)}
 & = \norm{F_{\eps}(S_n+\lambda I)^{\frac{1}{2}-\theta}q_m(S_n)
   (S+\lambda I)^{\frac{1}{2}} (S+\lambda I)^{-\frac{1}{2}} \paren{T_n^*
     \Y - S_n f^*_\cH}}\\
& \leq \norm{F_{\eps}(S_n+\lambda I)^{1-\theta}q_m(S_n)}
  \norm{(S_n+\lambda I)^{-\frac{1}{2}} (S+\lambda I)^{\frac{1}{2}}}  \delta(\lambda)\\
& \leq \Lambda \delta(\lambda) \paren{\sup_{x \in [0,\eps]} x^{1-\theta} q_m(x)
  + \lambda^{1-\theta} \sup_{x \in [0,\eps]} q_m(x)}\\
& \leq \Lambda \delta(\lambda) \paren{ 
\paren{\sup_{x \in [0,\eps]} q_m(x)}^{\theta}
\paren{\sup_{x \in [0,\eps]} x q_m(x)}^{1-\theta}
  + \lambda^{1-\theta} \abs{p'_m(0)}}\\
& \leq \Lambda \delta(\lambda) \paren{ \abs{p'_m(0)}^\theta + \lambda^{1-\theta} \abs{p'_m(0)}}\\
& \leq 2 \Lambda \delta(\lambda) \wt{\eps}^{-1} 
\paren{ \lambda + \wt{\eps}}^{1-\theta}\,.
\end{align*}
The second to last inequality is obtained by the following argument: if $m\geq 1$\,,
since $\eps\leq x_{1,m}$\,, $p_m$ is decreasing and convex in $[0,\eps]$\,
(see Lemma~\ref{lem:orthpol}, point (ii) ), we have
\[
q_m(x) = \frac{1-p_m(x)}{x} \leq \abs{p'_m(0)} \qquad \text{ for } x\in[0,\eps]\,;
\]
and also $ xq_m(x) = 1-p_m(x) \leq 1$ for $x\in[0,\eps]\,.$ If $m=0$\,,
we have $p_0 \equiv 1$ and $q_m\equiv 0$\,, so that $f_m=\bar{f}_m=0$
and the above upper bound is also trivially satisfied
for any $\eps>0$\,.\\
{\em Second summand:} 
Using \eqref{eq:boundin2},
and the fact that $\abs{p_m(x)} \leq 1$ for $x\in[0, \eps]$\,:
\begin{align*}
(II) = \norm{F_{\eps} (S_n+\lambda I)^{\frac{1}{2}-\theta}(\bar{f}_m-f^*_\cH)}
& =  \norm{F_{\eps} (S_n+\lambda I)^{\frac{1}{2}-\theta}p_m(S_n)S^\mu w} \\
& \leq \Lambda^{2} \paren{\eps^\mu (\eps+\lambda)^{\frac{1}{2}-\theta}
+ c(\mu) Z_\mu(\lambda) (\eps+\lambda)^{\frac{1}{2}-\theta}}\norm{w}\\
& \leq \Lambda^{2} \paren{\eps^\mu + c(\mu) Z_\mu(\lambda)}(\eps+\lambda)^{\frac{1}{2}-\theta}
\kappa^{-\mu-\frac{1}{2}} \rho.
\end{align*}
{\em Third summand:} observe that since $F_{\eps}^\perp = \mbf{1}_{[\eps,\infty)}(S_n)$\,, 
we can write $F_{\eps}^\perp = F_{\eps}^\perp S_n^{-1}S_n$ and
\begin{align*}
(III) = \norm{F_{\eps}^\perp (S_n+\lambda I)^{\frac{1}{2}-\theta}(f_m-f^*_\cH)}
& \leq \norm{F_{\eps}^\perp (S_n+\lambda I)^{1-\theta}S_n^{-1}}
\norm{  F_{\eps}^\perp(S_n+\lambda I)^{-\frac{1}{2}}S_n(f_m-f^*_\cH)}\\
&\begin{aligned}
 \leq  
\paren{\eps + \lambda}^{1-\theta} \eps^{-1}
& \left(
\norm{F_{\eps}^\perp (S_n+\lambda I)^{-\frac{1}{2}}T_n^*(T_nf_m -\Y)}
\right.\\
&\left. \;\; +
\norm{(S_n+\lambda I)^{-\frac{1}{2}}(T_n^*\Y - S_n f^*_\cH)}\right)
\end{aligned}\\
&\leq \eps^{-1} \paren{\eps+\lambda}^{\frac{1}{2}-\theta}
\norm{T_n^*(T_nf_m -\Y)} +
\Lambda 
\eps^{-1} \paren{\eps + \lambda}^{1-\theta} 
\delta(\lambda)\,.
\end{align*}
Gathering the three terms and rearranging leads to the announced inequality.
\end{proof}

\begin{lemma}[Bounding the residue]
\label{le:lemma1}
Assume condition  {\bf SC($r,\rho$)} holds, $r\geq \frac{1}{2}$\,.
Let $\lambda>0$\, be fixed and assume event {\bf B($\lambda$)} holds.
Then for any iteration step $1\leq m\leq n_{\Y}$\,:
\begin{align}
\norm{T_n^*(T_n f_m -\Y)}  \leq & c(\mu) \Lambda^2 \paren{ \abs{p'_m(0)}^{-(\mu+1)}
  + Z_\mu(\lambda)\abs{p'_m(0)}^{-1}}
\kappa^{-\mu-\frac{1}{2}}\rho \nonumber \\
& + \paren{\abs{p'_m(0)}^{-\frac{1}{2}}+\lambda^{\frac{1}{2}}}\Lambda \delta(\lambda)\,.
\label{eq:lemma1}
\end{align}
\end{lemma}
\begin{proof}
Using \eqref{eq:Nem1} of Lemma~\ref{lem:orthpol} and the notation therein,
it holds
\begin{align*}
\norm{T_n^*(T_n f_m -\Y)} = \norm{p_m(S_n)T_n^*\Y}
& \leq \norm{F_{x_{1,m}} \varphi_m(S_n) T_n^* \Y }\\
& \leq \norm{F_{x_{1,m}} \varphi_m(S_n) S_n f^*_\cH} + \norm{F_{x_{1,m}} \varphi_m(S_n)
(T_n^* \Y - S_n f^*_\cH)}\\
& := (I) + (II).
\end{align*}
We start with controlling the second term:
\begin{align*}
(II) = \norm{F_{x_{1,m}} \varphi_m(S_n) (T_n^* \Y - S_n f^*_\cH)}
& = \norm{F_{x_{1,m}} \varphi_m(S_n) (S+\lambda I)^{\frac{1}{2}}
  (S+\lambda I)^{-\frac{1}{2}}(T_n^* \Y - S_n f^*_\cH)}  \\
& \leq \norm{F_{x_{1,m}} \varphi_m(S_n) (S_n+\lambda I)^{\frac{1}{2}}}
\Lambda \delta(\lambda)\\
& \leq \paren{\sup_{x\in[0, x_{1,m}]} x^{\frac{1}{2}} \varphi_m(x) +
\lambda^{\frac{1}{2}}\sup_{x\in[0, x_{1,m}]} \varphi_m(x)} \Lambda
\delta(\lambda)\\
& \leq \paren{\abs{p'_m(0)}^{-\frac{1}{2}} + \lambda^{\frac{1}{2}}}
\Lambda \delta(\lambda)\,,
\end{align*}
where we used \eqref{eq:boundin0}, the first component of event {\bf B$(\lambda)$},
and in the last line 
inequality \eqref{eq:boundphi} with $\nu=0,1$\,.
For the first term, we use assumption {\bf SC($r,\rho$)}, then \eqref{eq:boundin2}:
\begin{align*}
(I) = \norm{F_{x_{1,m}} \varphi_m(S_n) S_n f^*_\cH} & =
\norm{F_{x_{1,m}} \varphi_m(S_n) S_n S^\mu w}\\
& \leq \Lambda^{2} \paren{\sup_{t\in[0,x_{1,m}]} t^{\mu+1} \varphi_m(t)
+ c(\mu) Z_\mu(\lambda) \sup_{t\in[0,x_{1,m}]} t \varphi_m(t)} \kappa^{-\mu-\frac{1}{2}}\rho\\
& \leq c(\mu) \Lambda^2 \paren{ \abs{p'_m(0)}^{-(\mu+1)} + Z_\mu(\lambda) \abs{p'_m(0)}^{-1}}
\kappa^{-\mu-\frac{1}{2}}\rho\,,
\end{align*}
where for the last inequality we applied \eqref{eq:boundphi} with $\nu = 2(\mu+1)$\,, $\nu=2$\,.
\end{proof}

We now consider the sequence of polynomials $p_m^{(2)}$ that are orthogonal 
with respect to the scalar product $[.,.]_{(2)}$ (see Lemma~\ref{lem:orthpol}, point (iv) ).
For notational convenience and compatibility below we define $x_{1,0}=x_{1,0}^{(2)}=\infty$\,.
\begin{lemma}
\label{le:relpol}
Assume condition  {\bf SC($r,\rho$)} holds, $r\geq\frac{1}{2}$\,.
For any $\lambda>0$\,, if the event {\bf B($\lambda$)} is satisfied,
then for any iteration step $1 \leq m \leq n_{\Y}$\,, and any
$\eps \in (0,x_{1,m-1})$\,:
\begin{align}
\brac{p_{m-1},p_{m-1}}_{(0)}^{\frac{1}{2}} &
= \norm{p_{m-1}(S_n)T_n^*\Y} \nonumber \\
& \leq \Lambda (\eps+\lambda)^\frac{1}{2}
\delta(\lambda) + c(\mu)\Lambda^2\eps\paren{\eps^\mu +
Z_\mu(\lambda)  }\kappa^{-\mu-\frac{1}{2}}\rho +
\eps^{-\frac{1}{2}}\brac{p_{m-1}^{(2)},p_{m-1}^{(2)}}^{\frac{1}{2}}_{(1)}\,.
\label{eq:le3}
\end{align}
\end{lemma}
\begin{proof}
By the optimality property defining the CG algorithm,
\begin{align*}
\norm{p_{m-1}(S_n)T_n^*\Y} 
\leq \norm{p^{(2)}_{m-1}(S_n)T_n^*\Y} & \leq \norm{F_{\eps}p^{(2)}_{m-1}(S_n)T_n^*\Y } + \norm{F_\eps^\perp
p^{(2)}_{m-1}(S_n)T_n^*\Y}\\
& \leq \norm{F_\eps T_n^* \Y} + \eps^{-\frac{1}{2}}
\norm{p^{(2)}_{m-1}(S_n)S_n^{\frac{1}{2}}T_n^*\Y} \\
&= \norm{F_\eps T_n^* \Y} +\eps^{-\frac{1}{2}}\brac{p_{m-1}^{(2)},p_{m-1}^{(2)}}^{\frac{1}{2}}_{(1)}
\end{align*}
For the last inequality, we have used the fact that
$|p^{(2)}_{m-1}(x)| \leq 1$ for $x \in [0,x_{m-1}^{(2)}]$\,,
(since $p^{(2)}_{m-1}(0)=1$ and $p^{(2)}_{m-1}$ is nonincreasing on
$[0,x_{m-1}^{(2)}]$, the case $m=1$ being trivial)
along with the assumption $0 < \eps < x_{1,m-1}$\,, as well as
$x_{1,m-1} \leq x_{1,m-1}^{(2)}$\, (for both of these properties see Lemma~\ref{lem:orthpol}, point (iv)).
We now bound
\begin{align*}
\norm{F_\eps T_n^* \Y } & \leq \norm{F_\eps (T_n^* \Y - S_n
  f^*_\cH)} + \norm{F_\eps S_n f^*_\cH} \\
& \leq \norm{F_\eps (S_n+\lambda I)^\frac{1}{2}}
\norm{(S_n+\lambda I)^{-\frac{1}{2}}(T_n^*\Y - S_n f^*_\cH)}
+ \norm{F_\eps S_n S^\mu w} \\
& \leq \Lambda(\eps+\lambda)^{\frac{1}{2}}\delta(\lambda) + \norm{F_\eps S_n S^\mu w}\,\\
& \leq \Lambda(\eps+\lambda)^{\frac{1}{2}}\delta(\lambda) + c(\mu)\Lambda^2\eps\paren{\eps^\mu +
Z_\mu(\lambda)  }\kappa^{-\mu-\frac{1}{2}}\rho,
\end{align*}
where we have used~\eqref{eq:boundin2} for the last line.
\end{proof}

{\bf Proof of Theorem~\ref{thm:inner}.}

We set 
\begin{equation}
\label{eq:lambdastar}
\blambda_* = \paren{\frac{4D \log 6 \gamma^{-1}}{\sqrt{n}}}^{\frac{2}{2\mu+s+1}}\,, 
\text{ and } \lambda_* := \kappa \blambda_*\,.
\end{equation}
(This normalization was introduced by \citealp{Cap10}.)
The assumed lower bound \eqref{eq:condninner} on $n$
ensures $\blambda_* \leq 1$\,.
We rewrite equivalently the discrepancy stopping rule as follows: for some fixed $\tau>0$\,,
\begin{equation}
\label{eq:apriorisr}
\wh{m} := \min\set{m \geq 0 : \norm{T_n^*(T_n f_m - \Y)} \leq
  (2+\tau)\lambda_*^{\frac{1}{2}} \delta(\lambda_*)}\,,
\end{equation}
where
\begin{equation}
\label{eq:deltastar}
\delta(\lambda_*):=  
\frac{3}{4} M \blambda_*^{\mu+\frac{1}{2}}\,.
\end{equation}
Observe that the above
$\tau>0$ is related from the constant $\tau'>3/2$ in \eqref{eq:disc2}
via $\tau = \frac{4}{3}(\tau' - \frac{3}{2})$\,.


We first check that event {\bf B($\lambda_*$)}, is satisfied with
large probability, using for this concentration
results which were recalled in Lemma~\ref{prop:relconc}.
From inequality~\eqref{eq:bybn}, with probability
$1-\gamma/3$\, we have
\begin{align}
\norm{(S+\lambda_* I)^{-\frac{1}{2}}(T_n^* \Y - S_n f^*_{\cH})} & \leq
 2M\paren{\sqrt{\frac{\cN(\lambda_*)}{n}} + \frac{2\sqrt{\kappa}}{\sqrt{\lambda_*}n}}
\log \frac{6}{\gamma} \nonumber \\
& \leq \frac{2M}{\sqrt{n}} D
\blambda_*^{-\frac{s}{2}} \paren{1 +
  \frac{1}{2D^2}\paren{\frac{4D}{\sqrt{n}} \log \frac{6}{\gamma}}
\blambda_*^{\frac{s-1}{2}}} \log
\frac{6}{\gamma} \nonumber \\
& \leq \frac{M}{2}
\blambda_*^{\mu+\frac{1}{2}} \paren{1+\frac{1}{2D^2}\blambda_*^{\mu+s}}
\nonumber \\
&  \leq \frac{3}{4} M \blambda_*^{\mu+\frac{1}{2}}= \delta(\lambda_*)\,,
\label{eq:defdelta1}
\end{align}
where we have used {\bf ED($s,D$)}, \eqref{eq:lambdastar} and the
assumptions $D\geq 1$ and $\blambda_* \leq 1$\,, as well as the fact that
$\log 6\gamma^{-1}\geq 1$\,. This ensures the first component
of {\bf B$(\lambda_*)$} is satisfied with probability $1-\gamma/3$\,.
We now turn to the second component. Inequality \eqref{eq:reloperror}
along with a repetition of the above reasoning yields that with
probability $1-\gamma/3$\,:
\[
\norm{(S+\lambda_* I)^{-\frac{1}{2}}(S_n-S)}_{HS} \leq
\frac{\sqrt{\kappa}}{M} \delta(\lambda_*)\,,
\]
so that
\[
\norm{(S+\lambda_*I)^{-1}(S_n-S)} \leq
\frac{\sqrt{\kappa}}{M} \lambda_*^{-\frac{1}{2}} \delta(\lambda_*)  = \frac{3}{4} \blambda_*^\mu \leq \frac{3}{4}\,.
\]
Observe that
\[
(S+\lambda_* I)(S_n+\lambda_* I)^{-1} = \paren{ (S_n-S)(S+\lambda_* I)^{-1} + I}^{-1}
\]
and use the inequality $\norm{(I-A)^{-1}} = \norm{\sum_{k\geq 0} A^k} \leq (1-\norm{A})^{-1}$ for $\norm{A} < 1$\,,
to obtain that the second component in {\bf B($\lambda_*$)} is satisfied
with $\Lambda:=2$\, (with probability $1-\gamma/3$). 

Finally, equation \eqref{eq:hoeffop} implies that the third component in {\bf B($\lambda_*$)}
is satisfied with probability $1-\gamma/3$\,, with
\begin{equation}
\label{eq:defDelta}
\Delta := 4 \sqrt{\frac{\log 6\gamma^{-1}}{n}}\,.
\end{equation}
To conclude, by the union bound, the three components of event {\bf B$(\lambda_*)$} are satisfied simultaneously 
with probability larger than $1-\gamma$\,. We assume for the rest of the proof that this event is satisfied.

The structure of the proof is now as follows: we aim at bounding the error of the estimator
using the inequality of Lemma~\ref{le:lemma2}. In this upper bound, the residue term
is controlled by definition of the stopping rule. 
The only and most difficult remaining quantity to control is then $\abs{p'_{\wh{m}}(0)}$\,.
Using Lemma~\ref{le:lemma1} on the residue at iteration $\wh{m}-1$\,,
and the definition of the stopping criterion, will allow to upper bound $\abs{p'_{\wh{m}-1}(0)}$\,;
finally Lemma~\ref{le:relpol} allows to relate iterations $\wh{m}-1$ and $\wh{m}$\,.

We will assume $\wh{m}\geq 1$\, for the remainder of the proof and postpone
to the end the (simpler) case $\wh{m}=0$\,.

{\bf First step:} upper bound on $\abs{p'_{\wh{m}-1}(0)}$\,.\\
 By definition of the stopping rule, we
have $\norm{T_n^*(T_n f_{\wh{m}-1} -\Y)} > (2 + \tau)
\lambda_*^\frac{1}{2} \delta(\lambda_*)$\,. Now applying this together
with the upper bound of Lemma \ref{le:lemma1} and rearranging, we get
\begin{align*}
\tau \lambda_*^\frac{1}{2} \delta(\lambda_*)
& \leq  c(\mu) \paren{\abs{p'_{\wh{m}-1}(0)}^{-(\mu+1)}
  + Z_\mu(\lambda_*)\abs{p'_{\wh{m}-1}(0)}^{-1}}
\kappa^{-\mu-\frac{1}{2}} \rho + 2 \abs{p'_{\wh{m}-1}(0)}^{-\frac{1}{2}} \delta(\lambda_*)\,\\
& \begin{aligned}
\leq c(\mu) \max & \left( \abs{p'_{\wh{m}-1}(0)}^{-\frac{1}{2}}
  \delta(\lambda_*), \rho \kappa^{-\mu-\frac{1}{2}}
    \abs{p'_{\wh{m}-1}(0)}^{-(\mu+1)}, \right.\\
&\left. \;\; \rho \kappa^{-\mu-\frac{1}{2}}
    Z_\mu(\lambda_*)\abs{p'_{\wh{m}-1}(0)}^{-1} \right)\,.
\end{aligned}
\end{align*}
We examine in succession the possibilities that the maximum in the above
expression is attained for each of the terms which comprise it. If the
first term attains the maximum, this implies $|p'_{\wh{m}-1}(0)| \leq c(\mu) \tau^{-2} \lambda_*^{-1}\,.$
If the second term attains the maximum, this entails
\[
c(\mu) \rho \kappa^{-\mu-\frac{1}{2}}
    \abs{p'_{\wh{m}-1}(0)}^{-(\mu+1)} \geq \tau \lambda_*^\frac{1}{2} \delta(\lambda_*)\,,
\]
which using \eqref{eq:deltastar} yields $
\abs{p'_{\wh{m}-1}(0)} \leq
c(\mu,\tau) \paren{\frac{\rho}{M}}^\frac{1}{\mu+1} \lambda_*^{-1}\,.
$
Finally, if the third term attains the maximum, we have
\[
c(\mu) \rho Z_\mu(\lambda_*) \kappa^{-\mu-\frac{1}{2}}
    \abs{p'_{\wh{m}-1}(0)}^{-1} \geq \tau \lambda_*^\frac{1}{2} \delta(\lambda_*)\,,
\]
which using \eqref{eq:deltastar} yields
$\abs{p'_{\wh{m}-1}(0)} \leq
c(\mu,\tau) \frac{\rho}{M} \lambda_*^{-\mu-1}
Z_\mu(\lambda_*)\,.$
We now establish the inequality
\begin{equation}
\label{eq:boundZ}
Z_\mu(\lambda_*) \lambda_*^{-\mu} \leq 1\,.
\end{equation}
The inequality is trivial if $\mu \leq 1$ given the definition of $Z_\mu(\lambda_*)$ in \eqref{eq:defZ}.
If $\mu >1$ holds, from the definition \eqref{eq:defDelta}, it holds that $\Delta \leq
\blambda_*^{\frac{2\mu+s+1}{2}}$\, (using $D\geq 1$\,, $\log 6\gamma^{-1}\geq 1$); hence
\[
Z_\mu(\lambda_*) \lambda_*^{-\mu} =
\Delta \blambda_*^{-\mu}
\leq \blambda_*^{\frac{s+1}{2}}
\leq 1\,.
\]
Gathering all three cases, we obtain that it always holds that
\begin{equation}
\label{eq:boundppmm1}
\abs{p'_{\wh{m}-1}(0)} \leq c(\mu,\tau) \max\paren{\frac{\rho}{M},1} \lambda_*^{-1}\,.
\end{equation}

{\bf Second step:} upper bound on $\abs{p'_{\wh{m}}(0)}$\,. For this we use the result of the first step and relate $\abs{p'_{\wh{m}-1}(0)}$ to $\abs{p'_{\wh{m}}(0)}$\,
using property \eqref{eq:pcompar} of orthogonal polynomials, which we recall here for convenience:
\begin{equation}
\label{eq:pcompar2}
\abs{{p_{m-1}}'(0) - {p_{m}}'(0)}
\leq \frac{\brac{p_{m-1},p_{m-1}}_{(0)}}{\brac{p^{(2)}_{m-1},p^{(2)}_{m-1}}_{(1)}}\,.
\end{equation}
To upper bound the above quantity, we apply Lemma \ref{le:relpol}
with the choice $\lambda=\lambda_*$ and
\newcommand{\epsos}{\eps_{0}}
\newcommand{\azero}{a_0}
\[
\eps = \epsos := \azero(\mu,\tau) \min \paren{\frac{M}{\rho},1} \lambda_*\,,
\]
where $0< \azero(\mu,\tau)\leq 1$ will be chosen small enough in order to satisfy some
constraints to be specified below. (We must insist here for the consistency of the argument
that contrarily to the notation $c(\ldots)$\,, the notation $\azero(\mu,\tau)$ denotes
a fixed value that does not change throughout the proof.) 
The first constraint is the
requirement $\epsos \in (0,x_{1,m-1})$ in order to apply Lemma~\ref{le:relpol}. For this, it can be seen from \eqref{eq:boundppmm1}
that $\azero(\mu,\tau)$ can be chosen small enough (namely smaller than
the inverse of the constant $c(\mu,\tau)$ of equation \eqref{eq:boundppmm1})\,, to ensure
\[
\epsos \leq \abs{p'_{m-1}(0)}^{-1} \leq x_{1,m-1}\,,
\]
the second inequality above is an easy consequence of the fact that $p_{m-1}$
is convex on $[0,x_{1,m-1}]$ and $p_{m-1}(0)=1$\,. We can now apply
Lemma~\ref{le:relpol} and use inequality \eqref{eq:le3}. We
turn to upper bound the following quantity appearing on the RHS of \eqref{eq:le3}:
\begin{multline}
\Lambda(\epsos+\lambda_*)^\frac{1}{2}
\delta(\lambda_*) + c(\mu)\Lambda^2\epsos\paren{\epsos^\mu +
  Z_\mu(\lambda_*)}\kappa^{-\mu-\frac{1}{2}}\rho\\
\begin{aligned}
& \leq 2(\azero(\mu,\tau)+1)^\frac{1}{2}\lambda_*^{\frac{1}{2}}\delta(\lambda_*)
+ c(\mu)\azero(\mu,\tau)\min \paren{{\rho},M} \lambda_*^{\frac{1}{2}} \blambda_*^{\mu+\frac{1}{2}}\\
& \leq  (c(\mu)\azero(\mu,\tau)+2)\lambda_*^{\frac{1}{2}}\delta(\lambda_*)\,,
\end{aligned}
\label{eq:beps}
\end{multline}
where we have used $\Lambda=2$\,, the definition \eqref{eq:deltastar} for
$\delta(\lambda_*)$ and inequality $Z_\mu(\lambda_*) \leq
\lambda_*^{\mu}$\,, see \eqref{eq:boundZ}\,.
Now, we can choose $\azero(\mu,\tau)$ small enough so that in addition to the previous constraint, the factor in the last display
satisfies $c(\mu)\azero(\mu,\tau)\leq \frac{\tau}{2}$\,.
The definition of
the stopping rule entails
\begin{equation}
\label{eq:resr}
\brac{p_{m-1},p_{m-1}}_{(0)}^{\frac{1}{2}} = \norm{T_n^*(T_n f_{\wh{m}-1} -\Y)} > (2 + \tau)
\lambda_*^\frac{1}{2} \delta(\lambda_*)\,.
\end{equation}
Now combining
\eqref{eq:le3}, \eqref{eq:resr} and \eqref{eq:beps}, we obtain
\begin{align*}
\brac{p_{m-1},p_{m-1}}_{(0)}^{\frac{1}{2}} 
&\leq (2 + \tau/2) \lambda_*^\frac{1}{2} \delta(\lambda_*) +
\epsos^{-\frac{1}{2}}\brac{p_{m-1}^{(2)},p_{m-1}^{(2)}}^{\frac{1}{2}}_{(1)}\\
&\leq \frac{2 + \tau/2}{2+\tau} \brac{p_{m-1},p_{m-1}}_{(0)}^{\frac{1}{2}} +
\epsos^{-\frac{1}{2}}\brac{p_{m-1}^{(2)},p_{m-1}^{(2)}}^{\frac{1}{2}}_{(1)}\,,
\end{align*}
so that
\[
\paren{2+4\tau^{-1}}^{-1}\brac{p_{m-1},p_{m-1}}_{(0)}^{\frac{1}{2}} \leq
\epsos^{-\frac{1}{2}} \brac{p_{m-1}^{(2)},p_{m-1}^{(2)}}_{(1)}^{\frac{1}{2}}\,;
\]
using this inequality in relation with \eqref{eq:pcompar2} and \eqref{eq:boundppmm1}, we obtain
\begin{equation}
\label{eq:boundpp}
\abs{p_{\wh{m}}'(0)}  \leq \abs{p_{\wh{m}-1}'(0)} + c(\tau)\epsos^{-1} \leq c(\mu,\tau) \max\paren{\frac{\rho}{M},1} \lambda_*^{-1}\,.
\end{equation}

{\bf Final step.} We want to apply the main error bound of Lemma~\ref{le:lemma2} with
$\lambda=\lambda_*$ and $\eps=\eps_* = a(\mu,\tau) \min \paren{\frac{M}{\rho},1} \lambda_*$\,.
Note that $\eps_*$ is different from $\epsos$ considered above; in fact $\eps_*$ must
now satisfy the constraint $\eps_* \in (0,x_{1,m})$ in order to be able to apply the lemma. 
In view of~\eqref{eq:boundpp}, we can choose $a(\mu,\tau) \in (0,1]$
small enough so as to ensure 
\[
\eps_* \leq \abs{p'_{m}(0)}^{-1} \leq x_{1,m}\,,
\]
similarly to the previous step (but now at iteration $m$ instead of $m-1$).
Recall that in the notation of Lemma~\ref{le:lemma2}, $\wt{\eps}_* = \min(\eps_*,\abs{p'_m(0)}^{-1})$\,,
so that with the above choice we have $\wt{\eps}_* = \eps_*$\,.
We now apply Lemma~\ref{le:lemma2}, plug in the inequality (by definition of the stopping rule)
\[
\norm{T_n^*(T_nf_{\wh{m}}-\Y)} \leq (2+\tau) \lambda_*^{\frac{1}{2}} \delta(\lambda_*)\,,
\]
and obtain, using again \eqref{eq:boundZ}:
\begin{align*}
\norm{S^{\frac{1}{2}-\theta}(f_m-f^*_\cH)} & \leq  c(\mu,\tau) \paren{
{\wt{\eps}_*}^{-1} \lambda_*^{1-\theta} \delta(\lambda_*)
+\lambda_*^{\frac{1}{2}-\theta+\mu}
\kappa^{-\mu-\frac{1}{2}} \rho + \eps_*^{-1}  \lambda_*^{1-\theta} 
\delta(\lambda_*)}\\
& \leq c(\mu,\tau) \paren{
\max\paren{\frac{\rho}{M},1}  \lambda_*^{-\theta} \delta(\lambda_*) 
+\lambda_*^{\frac{1}{2}-\theta+\mu} \kappa^{-\mu-\frac{1}{2}} \rho }\\
& \leq c(\mu,\tau)
\max\paren{\rho,M}  \lambda_*^{-\theta} \blambda_*^{\mu+\frac{1}{2}}\\
& = c(\mu,\tau)
\max\paren{\rho,M} \kappa^{-\theta} \paren{\frac{4D}{\sqrt{n}}\log \frac{6}{\gamma}}^{\frac{2\mu+1-2\theta}{2\mu+s+1}}\\
& = c(\mu,\tau)
\max\paren{\rho,M} \kappa^{-\theta} \paren{\frac{4D}{\sqrt{n}}\log \frac{6}{\gamma}}^{\frac{2(r-\theta)}{2r+s}}\,.
\end{align*}
If $\wh{m}=0$\,, we can apply directly Lemma \ref{le:lemma2} as above
without requiring the two previous steps, since
in this case $p'_0(0)=0$\,, so that we obtain the same final bound.

\subsection{Proof of Theorem~\ref{thm:outer}}
\label{proof:outer}

\newcommand{\fta}{f^\lambda_{\cH}}

In the case of the ``outer'' rates of convergence, i.e. condition {\bf SC}($r,\rho$)
holds with $r\in(0,\frac{1}{2})$\,, we recall that the target function $f^*$
is not representable as an element of the Hilbert space $\cH$\,. This means
many arguments used in Section~\ref{se:prinner} can't be used directly.
To alleviate this, we consider an approximation of $f^*$ by a function belonging
to $\cH$ defined as
\begin{equation}
\label{eq:fta}
\fta := (S_n + \lambda I)^{-1} T^* f^*\,.
\end{equation}
Similarly to the previous proof, we define an event where the estimation error is
controlled in an appropriate sense:
\newcommand{\wdel}{\wt{\delta}}
\[
\text{ \bf B'}(\lambda) :
\begin{cases}
\norm{(S+\lambda I)^{-\frac{1}{2}}(T_n^*\Y-Tf^*)} & \leq \delta(\lambda)\,,\\
\norm{(S+\lambda I)^{-\frac{1}{2}}(S_n-S)}_{HS} & \leq 
\wdel(\lambda)\,,\\
\displaystyle \norm{(S+\lambda I)(S_n
  +\lambda I)^{-1}}_{HS} & \leq \Lambda^2,
\end{cases}
\]
Notice that the first part of the event is slightly different from the corresponding
part of {\bf B}$(\lambda)$\,; this is because we will be using concentration inequality \eqref{eq:byby} rather than \eqref{eq:bybn}, the latter only being available for
$r\geq \frac{1}{2}$\,.

Our first lemma controls the approximation from $T\fta$ to the target $f^*$\,.

\begin{lemma}
\label{le:apr1}
Assume condition {\bf SC}$(r,\rho)$ holds, $r<\frac{1}{2}$\,.
Let $\theta$ be fixed, $\theta \in [ 0,r)$\,. 
For any $\lambda>0$\,, if the event {\bf B'($\lambda$)} is satisfied,
then
\[
\norm{K^{-\theta}(T\fta-f^*)} \leq \kappa^{-r} \rho \lambda^{r-\theta}
\paren{ 1 + \Lambda^2  \lambda^{-\frac{1}{2}}\wdel(\lambda)}\,,
\]
where $\fta$ is defined in \eqref{eq:fta}\,.
\end{lemma}
\begin{proof}
We first write 
\[
\norm{K^{-\theta}(T\fta-f^*)} \leq
\norm{K^{-\theta}T((S_n+\lambda)^{-1} - (S+\lambda)^{-1})T^*f^*}
+ \norm{K^{-\theta}(T(S+\lambda)^{-1}T^*-I)f^*}
\]
We focus on the second term first:
\begin{align*}
\norm{K^{-\theta}(T(S+\lambda I)^{-1}T^*-I)f^*}
& = \norm{K^{-\theta}(K(K+\lambda I)^{-1}-I)K^r u} \\
& \leq \kappa^{-r} \rho \lambda \sup_{t \in [0,\kappa]}
\frac{t^{r-\theta}}{t + \lambda}\\
& \leq \kappa^{-r} \rho \lambda^{r-\theta}\,,
\end{align*}
where at the last line we used Lemma~\ref{le:suptik}
and the assumption that $r \in (0,\frac{1}{2})$ and $\theta\in(0,r)$
so that $r-\theta \in (0,\frac{1}{2})$\,.

For the first term, we use the second component of {\bf B'}$(\lambda)$\,:
\begin{multline*}
\norm{K^{-\theta}T((S_n+\lambda I)^{-1} - (S+\lambda I)^{-1})T^*f^*}\\
\begin{aligned}
& = \norm{(K^{-\frac{1}{2}}T)S^{\frac{1}{2}-\theta}(S+\lambda I)^{-1}(S-S_n)(S_n+\lambda I)^{-1}T^*K^ru} \\
& \leq 
\norm{S^{\frac{1}{2}-\theta}(S+\lambda I)^{-\frac{1}{2}}}
\norm{(S+\lambda I)^{-\frac{1}{2}} (S-S_n)} \norm{(S_n+\lambda I)^{-1}T^*f^*} \\ 
& \leq \Lambda^2 \kappa^{-r} \rho \lambda^{r-\theta-\frac{1}{2}}\wdel(\lambda)\,;
\end{aligned}
\end{multline*}
for the last inequality, we bounded the last factor by
\begin{align}
\norm{(S_n+\lambda I)^{-1}T^*f^*} & 
\leq \norm{(S_n+\lambda I)^{-1} (S+\lambda I)} 
\norm{(S+\lambda I)^{-1}T^*K^ru} \nonumber \\
& \leq \Lambda^2 \norm{(S+\lambda I)^{-1}S^{r+\frac{1}{2}}(S^{-\frac{1}{2}}T)u} \nonumber \\
& \leq \Lambda^2 \kappa^{-r} \rho \sup_{t\in[0,\kappa]} \frac{t^{r+\frac{1}{2}}}{t+\lambda} \nonumber\\
& \leq \Lambda^2 \kappa^{-r} \rho \lambda^{r-\frac{1}{2}}\,, \label{eq:boundblah}
\end{align}
where we have used Lemma~\ref{le:suptik} again (since $r+\frac{1}{2}<1$)\,.
Collecting the terms yields the conclusion.
\end{proof}

\begin{lemma}[Bounding the error, outer case]
\label{le:lemma2o}
Assume condition  {\bf SC($r,\rho$)} holds, $r< \frac{1}{2}$\,.
For any $\lambda>0$\,, if the event {\bf B'($\lambda$)} is satisfied,
then for any $\theta \in [0, r)$\,, for any iteration step $1 \leq m \leq n_{\Y}$\,, for any
$\eps \in (0,x_{1,m})$\,, and denoting $\wt{\eps}:=\min(\eps,\abs{p'_m(0)}^{-1})$\,:
\begin{align*}
\norm{K^{-\theta}(Tf_m-f^*)} \leq & c(\Lambda) \Bigg( \eps^{-1}\paren{\eps+\lambda}^{\frac{1}{2}-\theta} \norm{T_n^*(T_nf_m-\Y)} 
 + \wt{\eps}^{-1}\paren{\lambda + \wt{\eps}}^{1-\theta}\delta(\lambda)\\
& \qquad \qquad + \kappa^{-r}\rho \paren{\lambda+\eps}^{r-\theta} \big(1 + \lambda^{-\frac{1}{2}}\wdel(\lambda)
+ \wt{\eps}^{-1}\lambda
\big) \Bigg)
\end{align*}
For $m=0$\,, the above inequality is valid for any $\eps>0$\,.
\end{lemma}
\begin{proof}
We begin with 
\[
\norm{K^{-\theta} (Tf_m-f^*)}  \leq 
\norm{K^{-\theta} T (f_m-\fta)} + \norm{K^{-\theta} (T\fta-f^*)}
\]
and the second term is dealt with by Lemma~\ref{le:apr1}. For the first term,
we will follow the proof of Lemma~\ref{le:lemma2} with appropriate changes.
\newcommand{\ftm}{\wt{f}_m}
Set $\ftm=q_m(S_n)T^*f^*$\,. We have
\begin{align*}
\norm{K^{-\theta} T (f_m-\fta)}& = 
\norm{S^{\frac{1}{2}-\theta}(f_m-\fta)}\\
& \leq \Lambda^{1-2\theta} \norm{(S_n+\lambda I)^{\frac{1}{2}-\theta}(f_m - \fta)}\\
&\leq
\Lambda\left(\norm{F_{\eps}(S_n+\lambda I)^{\frac{1}{2}-\theta}(f_m-\ftm)}
+ \norm{F_{\eps} (S_n+\lambda I)^{\frac{1}{2}-\theta}(\ftm-\fta)}\right. \\
& \qquad \; \; \left. + \norm{F_{\eps}^\perp (S_n+\lambda I)^{\frac{1}{2}-\theta}(f_m-\fta)} \right)\\
  &:= \Lambda ( (I) + (II) + (III) )\,,
\end{align*}
We upper bound the first summand, using the first component of event {\bf B'}($\lambda$):
\begin{multline*}
(I) = \norm{F_{\eps}(S_n+\lambda I)^{\frac{1}{2}-\theta}q_m(S_n)(T_n^*\Y-T^*f^*)} \\
\begin{aligned}
& \leq \Lambda \norm{F_{\eps}(S_n+\lambda I)^{1-\theta}q_m(S_n)} 
\norm{(S+\lambda I)^{-\frac{1}{2}}(T_n^*\Y-T^*f^*)} \\
& \leq 2 \Lambda \delta(\lambda) \abs{p'_m(0)} 
\paren{ \lambda + \abs{p'_m(0)}^{-1}}^{1-\theta}\,;
\end{aligned}
\end{multline*}
the above calculation is almost identical to the handling of term (I) in the
proof of Lemma~\ref{le:lemma2}, and we refer to that proof for the details.
We turn to the second term:
\begin{align*}
(II) = \norm{F_{\eps} (S_n+\lambda I)^{\frac{1}{2}-\theta}(\ftm-\fta)}
& = \norm{F_{\eps} (S_n+\lambda I)^{\frac{1}{2}-\theta}(q_m(S_n)(S_n+\lambda I) 
- I)(S_n+\lambda I )^{-1}T^*K^ru}\\
& \leq \norm{F_{\eps} (S_n+\lambda I)^{-\theta-\frac{1}{2}}(p_m(S_n) + \lambda q_m(S_n)
)S^{r+\frac{1}{2}}}\norm{S^{-\frac{1}{2}}T^*u}\\
& \leq
\Lambda^2 \kappa^{-r} \rho 
\sup_{t\in[0,\eps]} \paren{\abs{p_m(t)} + \lambda \abs{q_m(t)}}\paren{t+\lambda}^{r-\theta}\\
& \leq\Lambda^2 \kappa^{-r} \rho \paren{(\eps + \lambda)^{r-\theta} + \lambda^{1+r-\theta}\abs{p'_m(0)}
+ \lambda \abs{p'_m(0)}^{1+\theta-r}}\\
& \leq c(\Lambda) \kappa^{-r} \rho (\wt{\eps} + \lambda)^{r-\theta}\paren{ 1 + \wt{\eps}^{-1}\lambda}\,,
\end{align*}
where for the penultimate inequality, we used the same arguments as in the proof
of Lemma~\ref{le:lemma2} to bound the quantities involving $\abs{p_m(x)}$ and
$\abs{q_m(x)}$ on the interval $[0,\eps]\subset[0,x_{1,m}]$\,. We finally
consider the third term; we recall that we can write $F_{\eps}^\perp = F_{\eps}^\perp S_n^{-1}S_n$ and
\begin{align*}
(III) = \norm{F_{\eps}^\perp (S_n+\lambda I)^{\frac{1}{2}-\theta}(f_m-\fta)}
& \leq \norm{F_{\eps}^\perp (S_n+\lambda I)^{1-\theta} S_n^{-1}}
\norm{F_{\eps}^\perp (S_n+\lambda I)^{-\frac{1}{2}} S_n(f_m-\fta)}\\
& \leq \frac{\paren{\eps+\lambda}^{1-\theta}}{\eps}\left(
\norm{F_{\eps}^\perp (S_n+\lambda I)^{-\frac{1}{2}} T_n^*(T_nf_m-\Y)} \right.\\
& \qquad \qquad \qquad \qquad + \norm{ (S_n+\lambda I)^{-\frac{1}{2}} (T_n^* \Y - T^*f^*)}\\
& \qquad \qquad \qquad \qquad \left.
 + \norm{F_{\eps}^\perp (S_n+\lambda I)^{-\frac{1}{2}} (T^*f^*-S_n\fta)}\right)\\
& \leq \frac{\paren{\eps+\lambda}^{\frac{1}{2}-\theta}}{\eps}\norm{T_n^*(T_nf_m-\Y)}
+ \frac{\paren{\eps+\lambda}^{1-\theta}}{\eps}
\delta(\lambda) + (IV)\,,
\end{align*}
with
\begin{align*}
(IV) := \eps^{-1}(\eps+\lambda)^{1-\theta} \norm{F_{\eps}^\perp (S_n+\lambda I)^{-\frac{1}{2}} (T^*f^*-S_n\fta)}
& = \lambda \eps^{-1}(\eps+\lambda)^{1-\theta} 
\norm{F_{\eps}^\perp (S_n+\lambda I)^{-\frac{3}{2}} T^*f^*}\\
&  \leq \lambda \eps^{-1} (\eps+\lambda)^{\frac{1}{2}-\theta} \norm{(S_n+\lambda I)^{-1} T^*f^*}\\
& \leq \Lambda^2 \kappa^{-r}  \rho \eps^{-1}(\eps+\lambda)^{\frac{1}{2}-\theta} \lambda^{r+\frac{1}{2}}\\
& \leq \Lambda^2 \kappa^{-r}  \rho (\eps+\lambda)^{r-\theta} (1+\eps^{-1}\lambda) \,,
\end{align*}
where we have reused inequality \eqref{eq:boundblah} at the second to last line.
Gathering the different terms now yields the announced inequality.
\end{proof}

\begin{lemma}[Bounding the residue, outer case]
\label{le:lemma1o}
Assume condition  {\bf SC($r,\rho$)} holds, $r< \frac{1}{2}$\,.
Let $\lambda>0$\, be fixed and assume event {\bf B($\lambda$)} holds.
Then for any iteration step $1\leq m\leq n_{\Y}$\,:
\begin{align}
\norm{T_n^*(T_n f_m -\Y)}  \leq & \Lambda^2 \paren{ 2 \abs{p'_m(0)}^{-(r+\frac{1}{2})}
  + \lambda^{r+\frac{1}{2}}}
\kappa^{-r}\rho 
 + \Lambda \paren{\abs{p'_m(0)}^{-\frac{1}{2}}+\lambda^{\frac{1}{2}}} \delta(\lambda)\,.
\label{eq:lemma1o}
\end{align}
\end{lemma}
\begin{proof}
The proof is similar to that of Lemma~\ref{le:lemma1} up to the fact that
we use $T^*f^*$ instead of $S_n f^*_\cH$\,, so that we skip some details. The main inequality becomes
\[ \norm{T_n^*(T_n f_m -\Y)} 
 \leq \norm{F_{x_{1,m}} \varphi_m(S_n) T^* f^*} + \norm{F_{x_{1,m}} \varphi_m(S_n)
(T_n^* \Y - T^* f^*)} := (I) + (II),
\]
where we used \eqref{eq:Nem1} of Lemma~\ref{lem:orthpol} and the notation therein.
The second term is controlled exactly as in the proof of Lemma~\ref{le:lemma1},
only we use the first component of {\bf B'}{$(\lambda)$} instead of that of {\bf B}{$(\lambda)$}.
It gives rise to
\[
(II) \leq 
\paren{\abs{p'_m(0)}^{-\frac{1}{2}} + \lambda^{\frac{1}{2}}}
\Lambda \delta(\lambda)\,.
\]
For the first term, we use assumption {\bf SC($r,\rho$)}, then \eqref{eq:boundin2}
with $r<\frac{1}{2}$\,:
\begin{align*}
(I) = \norm{F_{x_{1,m}} \varphi_m(S_n) T^* f^*} & =
\norm{F_{x_{1,m}} \varphi_m(S_n) S^{r+\frac{1}{2}} (S^{-\frac{1}{2}}T^*)   u}\\
& \leq \Lambda^{2} \paren{\sup_{t\in[0,x_{1,m}]} t^{r+\frac{1}{2}} \varphi_m(t)
+ \lambda^{r+\frac{1}{2}} \sup_{t\in[0,x_{1,m}]} \varphi_m(t)} \kappa^{-r}\rho\\
& \leq \Lambda^2 \paren{ 2  \abs{p'_m(0)}^{-(r+\frac{1}{2})} + \lambda^{r+\frac{1}{2}} }
\kappa^{-r}\rho\,,
\end{align*}
where for the last inequality we applied \eqref{eq:boundphi} with $\nu = 2r+1 \leq 2$\,, $\nu=0$\,.
\end{proof}

Finally, the following lemma is the counterpart of Lemma~\ref{le:relpol} in the outer case:
\begin{lemma}
\label{le:relpolo}
Assume condition  {\bf SC($r,\rho$)} holds, $r<\frac{1}{2}$\,.
For any $\lambda>0$\,, if the event {\bf B($\lambda$)} is satisfied,
then for any iteration step $1 \leq m \leq n_{\Y}$\,, and any
$\eps \in (0,x_{1,m-1})$\,:
\begin{align}
\brac{p_{m-1},p_{m-1}}_{(0)}^{\frac{1}{2}} &
= \norm{p_{m-1}(S_n)T_n^*\Y} \nonumber \\
& \leq \Lambda (\eps+\lambda)^\frac{1}{2}
\delta(\lambda) + \Lambda^2\paren{\eps^{r+\frac{1}{2}} +
\lambda^{r+\frac{1}{2}}  }\kappa^{-r}\rho +
\eps^{-\frac{1}{2}}\brac{p_{m-1}^{(2)},p_{m-1}^{(2)}}^{\frac{1}{2}}_{(1)}\,.
\label{eq:le3o}
\end{align}
\end{lemma}
\begin{proof}
The first step of the proof is unchanged with respect to that of Lemma~\ref{le:relpol},
and we refer to it for the details:
\[
\norm{p_{m-1}(S_n)T_n^*\Y} 
\leq \norm{F_\eps T_n^* \Y} +\eps^{-\frac{1}{2}}\brac{p_{m-1}^{(2)},p_{m-1}^{(2)}}^{\frac{1}{2}}_{(1)}\,.
\]
Then we follow again the proof of Lemma~\ref{le:relpol}, but using $T^*f^*$ in place of $S_nf^*_\cH$\,:
\begin{align*}
\norm{F_\eps T_n^* \Y }  \leq \norm{F_\eps (T_n^* \Y - T^*
  f^*)} + \norm{F_\eps T^* f^*} 
%
& \leq \Lambda(\eps+\lambda)^{\frac{1}{2}}\delta(\lambda) 
+ \norm{F_\eps S^{r+\frac{1}{2}} (S^{-\frac{1}{2}} T^*) u}\,\\
& \leq \Lambda(\eps+\lambda)^{\frac{1}{2}}\delta(\lambda) + \Lambda^2\paren{\eps^{r+\frac{1}{2}} +
\lambda^{r+\frac{1}{2}}  }\kappa^{-r}\rho,
\end{align*}
where we have used~\eqref{eq:boundin2} (with $r<\frac{1}{2}$) for the last line.
\end{proof}
We provide for completeness the following simple lemma, which was used 
a couple of times:
\begin{lemma}
\label{le:suptik}
Let $\lambda>0$ and $\nu\in[0,1]$ be fixed. Then
\[
\sup_{t \in \mbr_+} \frac{t^\nu}{t+\lambda} = C(\nu) \lambda^{\nu-1},
\]
with $C(\nu)=\nu^\nu(1-\nu)^{(1-\nu)}\in [\frac{1}{2},1]$ if $\nu\in(0,1)$\,, and $C(0)=1$\,, $C(1)=1$\,.
\end{lemma}
\begin{proof}
If $\nu\in(0,1)$\,, the derivative of the function is equal to
$t^{\nu-1}((\nu-1)t + \nu \lambda)/(t+\lambda^2)$\,. The value
$t^*:=\nu\lambda/(1-\nu)$ is the position of the unique maximum on
$\mbr_+$\,, giving rise to the result. The special cases $\nu=0,1$ are
treated easily. Alternatively, the upper bound resulting from 
$C(\nu)\leq 1$ can be obtained more directly by using the inequality
$(t+\lambda)\geq t^\nu\lambda^{1-\nu}$\,. 
\end{proof}

{\bf Proof of Theorem~\ref{thm:outer}}

We fix the following values for $\lambda_*,\blambda_*$ similarly to the inner rate case:
\begin{equation}
\label{eq:lambdastaro}
\blambda_* = \paren{\frac{4D}{\sqrt{n}} \log \frac{ 4}{ \gamma}}^{\frac{2}{2r+s}}\,, 
\text{ and } \lambda_* := \kappa \blambda_*\,,
\end{equation}
satisfying $\blambda_*\leq 1$ because of assumption \eqref{eq:condnouter}\,.
The discrepancy stopping rule in the outer case can be rewritten as follows: for some fixed $\tau>0$\,,
\begin{equation}
\label{eq:apriorisro}
\wh{m} := \min\set{m \geq 0 : \norm{T_n^*(T_n f_m - \Y)} \leq
  (8+\tau)\max\paren{1,\frac{\rho}{M}}\lambda_*^{\frac{1}{2}} \delta(\lambda_*)}\,,
\end{equation}
where 
\begin{equation}
  \label{eq:deltastaro}
\delta(\lambda_*):=  
\frac{3}{4} M \blambda_*^{r}\,.
\end{equation}
(Observe that $\tau'>6$ from \eqref{eq:outthr} and $\tau'>0$ in \eqref{eq:apriorisro}
are related via $\tau = \frac{4}{3}(\tau' - 6)$\,.)


We check that event {\bf B'($\lambda_*$)}, is satisfied with
large probability.
To check the first component, we use \eqref{eq:byby}
instead of \eqref{eq:bybn}. Since the easily checked relation
$T_{\wtn}^*\wY = T^*_{n} \Y$ holds,
we have with probability $1-\gamma/2$\,,
\begin{equation}
\norm{(S+\lambda_* I)^{-\frac{1}{2}}(T_{\wtn}^* \wY - T^* f^*)}
 = \norm{(S+\lambda_* I)^{-\frac{1}{2}}(T_n^* \Y - T^* f^*)}
\leq \frac{3}{4} M \blambda_*^{\mu+\frac{1}{2}}= \delta(\lambda_*)\,,
\label{eq:defdelta1o}
\end{equation}
where this inequality follows from identical steps as for~\eqref{eq:defdelta1},
to which we refer for details (remember the notation $\mu=r-{\frac{1}{2}}$\,, there).
It is worth noting that in order for the argument leading to~\eqref{eq:defdelta1}
to be valid, we need to use the assumption $r+s\geq \frac{1}{2}$\,.
This ensures the first component
of {\bf B'$(\lambda_*)$} is satisfied with probability $1-\gamma/2$\,.
We now turn to the second component. We can apply the deviation inequality \eqref{eq:reloperror}
but with $n$ replaced by $\wt{n}$\,, since we make use of all the unlabeled data. 
Using the fact that $\frac{n}{\tilde n} \leq n^{-\frac{1-2r}{2r+s}} \leq {\blambda_*}^{1-2r}$\,,
we obtain that with probability $1-\gamma/2$\,:
\begin{align*}
\norm{(S+\lambda_* I)^{-\frac{1}{2}}(S_{\wtn}-S)}_{HS} 
& \leq
 2\sqrt{\kappa}
\paren{\sqrt{\frac{\cN(\lambda_*)}{\wtn}} + \frac{2\sqrt{\kappa}}{\sqrt{\lambda_*}\wtn}}
\log \frac{4}{\gamma} \\
& \leq
 \frac{\sqrt{\kappa}}{2} \paren{\frac{4D}{\sqrt{n}} \log \frac{4}{\gamma}} 
\paren{\blambda_*^{-\frac{s}{2}+\frac{1}{2}-r} 
+  \frac{1}{2D^2}\paren{\frac{4D}{\sqrt{n}} \log \frac{4}{\gamma}}
\blambda_*^{\frac{1}{2}-2r}} \\
& \leq \frac{\sqrt{\kappa}}{2}
\blambda_*^{\frac{1}{2}} \paren{1+\frac{1}{2}\blambda_*^{s}} \\
&  \leq \frac{3}{4} \sqrt{\kappa} \blambda_*^{\frac{1}{2}} =: \wdel(\lambda_*)\,,
\end{align*}
so that the second component of {\bf B'($\lambda_*$)} is satisfied with the 
above value for $\wdel(\lambda)$\,; moreover
\[
\norm{(S+\lambda_*I)^{-1}(S_n-S)} \leq
\lambda_*^{-\frac{1}{2}} \wdel(\lambda_*)  = \frac{3}{4} \,,
\]
implying (by the same argument as in the proof of Theorem~\ref{thm:inner}) that the third
compoment of {\bf B'($\lambda_*$)} is satisfied with $\Lambda:=2$\,. We can observe in
passing that obtaining the above inequality was the technical reason for introducing
the additional unlabeled data in the outer case, 
since using the labeled data alone would not have
granted it for this choice of $\lambda_*$\,.
Summarizing, the three components of event {\bf B'$(\lambda_*)$} are satisfied simultaneously 
with probability larger than $1-\gamma$\,. 
We assume for the rest of the proof that this event is satisfied.

We assume $\wh{m}\geq 1$\, for the remainder of the proof and postpone
to the end the (simpler) case $\wh{m}=0$\,.

{\bf First step:} upper bound on $\abs{p'_{\wh{m}-1}(0)}$\,.\\
 By definition of the stopping rule, we
have $\norm{T_n^*(T_n f_{\wh{m}-1} -\Y)} > (8 + \tau) \max(1,\frac{\rho}{M})
\lambda_*^\frac{1}{2} \delta(\lambda_*)$\,. Applying this together
with the upper bound of Lemma \ref{le:lemma1o},
observing that \eqref{eq:lambdastaro} entails 
$\Lambda^2 \lambda_*^{r+\frac{1}{2}} \kappa^{-r} \rho \leq 6 \rho 
\lambda_*^{\frac{1}{2}} \delta(\lambda_*)$\,, and rearranging, we get
\begin{align*}
\max(1,\frac{\rho}{M}) \tau \lambda_*^\frac{1}{2} \delta(\lambda_*)
& \leq  8 \abs{p'_{\wh{m}-1}(0)}^{-(r+\frac{1}{2})}
\kappa^{-r}\rho + 2 \abs{p'_{\wh{m}-1}(0)}^{-\frac{1}{2}} \delta(\lambda_*)\,\\
& 
\leq 10 \max
\paren{\abs{p'_{\wh{m}-1}(0)}^{-\frac{1}{2}} \delta(\lambda_*),
\abs{p'_{\wh{m}-1}(0)}^{-(r+\frac{1}{2})} \kappa^{-r}\rho}\, 
\end{align*}
If the maximum on the RHS is attained for the first term, 
this implies \\ $|p'_{\wh{m}-1}(0)| \leq 4\min(1,\frac{M}{\rho}) \tau^{-2} \lambda_*^{-1}\,.$
If the second term attains the maximum, this entails via \eqref{eq:deltastaro}
\[
10 \rho \kappa^{-r}
    \abs{p'_{\wh{m}-1}(0)}^{-(r+\frac{1}{2})} \geq \tau \max(1,\frac{\rho}{M}) 
\lambda_*^\frac{1}{2} \delta(\lambda_*) = 
\frac{3}{4}\tau \max(M,\rho) \kappa^{-r} \lambda_*^{r+\frac{1}{2}}\,,
\]
so that
\[
\abs{p'_{\wh{m}-1}(0)} \leq
c(r,\tau) \min(1,\frac{\rho}{M})^\frac{1}{r+\frac{1}{2}} \lambda_*^{-1}\,.
\]
Gathering the two cases, we obtain that it always holds that
\begin{equation}
\label{eq:boundppmm1o}
\abs{p'_{\wh{m}-1}(0)} \leq c(r,\tau) \lambda_*^{-1}\,.
\end{equation}

{\bf Second step:} upper bound on $\abs{p'_{\wh{m}}(0)}$\,. 
Apply Lemma \ref{le:relpolo}
with the choice $\lambda=\lambda_*$ and
\[
\eps = \epsos := \azero(r,\tau) \lambda_*\,,
\]
where $0< \azero(r,\tau)\leq 1$ will be chosen small enough in order to satisfy some
constraints to be specified. The first constraint is the
requirement $\epsos \in (0,x_{1,m-1})$ in order to apply Lemma
\ref{le:relpolo}. For this, it can be seen from \eqref{eq:boundppmm1o}
that $\azero(r,\tau)$ can be chosen small enough to ensure
\[
\epsos \leq \abs{p'_{m-1}(0)}^{-1} \leq x_{1,m-1}\,.
\]
We can now apply Lemma~\ref{le:relpolo}. We upper bound the following quantity 
appearing on the RHS of \eqref{eq:le3o}:
\begin{align}
\Lambda (\eps_0+\lambda_*)^\frac{1}{2}
\delta(\lambda_*) + \Lambda^2\paren{\eps_0^{r+\frac{1}{2}} +
\lambda_*^{r+\frac{1}{2}}  }\kappa^{-r}\rho \nonumber
& \leq 2(\azero(r,\tau)+1)\lambda_*^{\frac{1}{2}}\delta(\lambda_*)
+ 4(\azero(r,\tau)^{r+\frac{1}{2}}+1) \lambda_*^{\frac{1}{2}} \blambda_*^r \rho \nonumber\\
& \leq  (8+8\azero(r,\tau)^{\frac{1}{2}})
\max(1,\frac{\rho}{M})\lambda_*^{\frac{1}{2}}\delta(\lambda_*)\,, \label{eq:bepso}
\end{align}
We can choose $\azero(r,\tau)$ small enough so that in addition to the previous constraint, 
it satisfies $\azero(\mu,\tau)^\frac{1}{2}\leq \frac{\tau}{16}$\,.
Remember that the definition of
the stopping rule entails
\begin{equation}
\label{eq:resro}
\brac{p_{m-1},p_{m-1}}_{(0)}^{\frac{1}{2}} = \norm{T_n^*(T_n f_{\wh{m}-1} -\Y)} > (8 + \tau)
\max(1,\frac{\rho}{M})\lambda_*^\frac{1}{2} \delta(\lambda_*)\,.
\end{equation}
Combining
\eqref{eq:le3o}, \eqref{eq:resro} and \eqref{eq:bepso}, we obtain
\begin{align*}
\brac{p_{m-1},p_{m-1}}_{(0)}^{\frac{1}{2}} 
&\leq (8 + \tau/2) \max(1,\frac{\rho}{M}) \lambda_*^\frac{1}{2} \delta(\lambda_*) +
\epsos^{-\frac{1}{2}}\brac{p_{m-1}^{(2)},p_{m-1}^{(2)}}^{\frac{1}{2}}_{(1)}\\
&\leq \frac{8 + \tau/2}{8+\tau} \brac{p_{m-1},p_{m-1}}_{(0)}^{\frac{1}{2}} +
\epsos^{-\frac{1}{2}}\brac{p_{m-1}^{(2)},p_{m-1}^{(2)}}^{\frac{1}{2}}_{(1)}
\end{align*}
so that
\[
\paren{2+16\tau^{-1}}^{-1}\brac{p_{m-1},p_{m-1}}_{(0)}^{\frac{1}{2}} \leq
\epsos^{-\frac{1}{2}} \brac{p_{m-1}^{(2)},p_{m-1}^{(2)}}_{(1)}^{\frac{1}{2}}\,;
\]
using this inequality in relation with \eqref{eq:pcompar} and \eqref{eq:boundppmm1o}, 
we obtain
\begin{equation}
\label{eq:boundppo}
\abs{p_{\wh{m}}'(0)}  \leq \abs{p_{\wh{m}-1}'(0)} + c(\tau)\epsos^{-1} \leq 
c(r,\tau) \lambda_*^{-1}\,.
\end{equation}

{\bf Final step.} We want to apply the main error bound of Lemma~\ref{le:lemma2} with
$\lambda=\lambda_*$ and $\eps=\eps_* = a(r,\tau) \lambda_*$\,.
In view of~\eqref{eq:boundppo}, we can choose $a(\mu,\tau) \in (0,1]$
small enough so that to ensure 
\[
\eps_* \leq \abs{p'_{m}(0)}^{-1} \leq x_{1,m}\,.
\]
Recall that in the notation of Lemma~\ref{le:lemma2o}, 
$\wt{\eps}_* = \min(\eps_*,\abs{p'_m(0)}^{-1})$\,,
so that with the above choice we have $\wt{\eps}_* = \eps_*$\,.
We now apply Lemma~\ref{le:lemma2o}, plug in the inequality (by definition of the stopping rule)
\[
\norm{T_n^*(T_nf_{\wh{m}}-\Y)} \leq (8+\tau) \max(1,\frac{\rho}{M}) 
\lambda_*^{\frac{1}{2}} \delta(\lambda_*)\,,
\]
to obtain:
\begin{align*}
\norm{K^{-\theta}(f_\wh{m}-f^*)} \leq & c(r,\tau) 
\Bigg( \eps_*^{-1}\paren{\eps_*+\lambda_*}^{\frac{1}{2}-\theta} \max(1,\frac{\rho}{M}) 
\lambda_*^{\frac{1}{2}} \delta(\lambda_*)
 + \wt{\eps}_*^{-1}\paren{\lambda_* + \wt{\eps}_*}^{1-\theta}\delta(\lambda_*)\\
& \qquad \qquad + \kappa^{-r}\rho \paren{\lambda_*+\eps_*}^{r-\theta} 
\big(1 + \lambda_*^{-\frac{1}{2}}\wdel(\lambda) + \wt{\eps}_*^{-1}\lambda_*
\big) \Bigg)\\
& \leq c(r,\tau)\max(\rho,M)\kappa^{-\theta}\lambda_*^{r-\theta}\\
& = c(r,\tau)
\max\paren{\rho,M} \kappa^{-\theta} 
\paren{\frac{4D}{\sqrt{n}}\log\frac{4}{\gamma}}^{\frac{2(r-\theta)}{2r+s}}\,.
\end{align*}
If $\wh{m}=0$\,, we can apply directly Lemma \ref{le:lemma2o} as above
without requiring the two previous steps, since
in this case $p'_0(0)=0$\,, so that we obtain the same final bound.

\bibliographystyle{plainnat}
\bibliography{kcg_journal}

\end{document}